\documentclass[reqno]{amsart}
\usepackage{amsmath,amsfonts,amssymb,amsthm}
\usepackage[usenames,dvipsnames]{xcolor}
\usepackage{graphicx}
\usepackage[colorlinks=true]{hyperref} 
\usepackage{url}
\hypersetup{linkcolor=Violet,citecolor=PineGreen}
\newtheorem{teor}{Theorem}[section]

\newtheorem{prop}[teor]{Proposition}
\newtheorem{coro}[teor]{Corollary}
\theoremstyle{definition}
\newtheorem{defi}[teor]{Definition}

\newtheorem{hipos}[teor]{Hypotheses}
\newtheorem{nota}[teor]{Remark}

\newtheorem{notas}[teor]{Remarks}
\numberwithin{equation}{section}

\newcommand{\M}{\mathbb{M}}
\newcommand{\N}{\mathbb{N}}

\newcommand{\R}{\mathbb{R}}
\renewcommand{\S}{\mathbb{S}}
\renewcommand{\P}{\mathbb{P}}
\newcommand{\T}{\mathbb{T}}
\newcommand{\Y}{\mathbb{Y}}
\newcommand{\Z}{\mathbb{Z}}
\newcommand{\mA}{\mathcal{A}}
\newcommand{\mB}{\mathcal{B}}
\newcommand{\mC}{\mathcal{C}}
\newcommand{\mD}{\mathcal{D}}

\newcommand{\mK}{\mathcal{K}}

\newcommand{\mM}{\mathcal{M}}

\newcommand{\mP}{\mathcal{P}}

\newcommand{\mR}{\mathcal{R}}
\newcommand{\mS}{\mathcal{S}}
\newcommand{\mU}{\mathcal{U}}
\newcommand{\mV}{\mathcal{V}}

\newcommand{\mkp}{\mathfrak{p}}
\newcommand{\ep}{\varepsilon}
\newcommand{\pu}{{\cdot}}

\newcommand{\W}{\Omega}
\newcommand{\WS}{\W\!\times\!\S}
\newcommand{\WP}{\W\!\times\!\P}
\newcommand{\w}{\omega}
\newcommand{\wth}{\omega,\theta}
\newcommand{\wt}{\w\pu t}
\newcommand{\ws}{\w\pu s}

\newcommand{\bw}{\mathbf w}

\newcommand{\by}{\mathbf y}
\newcommand{\bz}{\mathbf z}

\newcommand{\bcero}{\mathbf 0}
\newcommand{\wit}{\widetilde}
\newcommand{\wih}{\widehat}

\newcommand{\nbd}{\nobreakdash}

\newcommand{\p}{\partial}
\newcommand{\lsm}{\left[\begin{smallmatrix}}
\newcommand{\rsm}{\end{smallmatrix}\right]}

\newcommand{\dist}{\text{\rm dist}}
\DeclareMathOperator{\tr}{tr}

\DeclareMathOperator{\cls}{closure}

\begin{document}
\title[Li-Yorke chaos in nonautonomous Hopf bifurcation.]
{Li-Yorke chaos in nonautonomous Hopf bifurcation patterns - I.}
\author[C.~N\'{u}\~{n}ez]{Carmen N\'{u}\~{n}ez}
\author[R.~Obaya]{Rafael Obaya}
\address
{Departamento de Matem\'{a}tica Aplicada, Universidad de
Valladolid, Paseo del Cauce 59, 47011 Valladolid, Spain}
\email[Carmen N\'{u}\~{n}ez]{carnun@wmatem.eis.uva.es}
\email[Rafael Obaya]{rafoba@wmatem.eis.uva.es}
\thanks{Partly supported by MINECO/FEDER (Spain)
under project MTM2015-66330-P and by European Commission under project
H2020-MSCA-ITN-2014.}
\date{}
\begin{abstract}
We analyze the characteristics of the global attractor of a type of
dissipative nonautonomous dynamical systems
in terms of the Sacker and Sell spectrum of its linear part.
The model gives rise to a pattern of nonautonomous Hopf bifurcation
which can be understood as a generalization of the classical autonomous one.
We pay special attention to the dynamics at the bifurcation point,
showing the possibility of occurrence of Li-Yorke chaos in the
corresponding attractor and hence of a high degree of unpredictability.
\end{abstract}
\maketitle
\section{Introduction}
The nonautonomous bifurcation theory
\textcolor{black}{
for ordinary or partial differential equations}
is a
\textcolor{black}{relatively new},
complex, and challenging subject
of study for which many fundamental problems remain open. One
of the initial questions is to determine what kind of objects
are those whose structural variation determines the occurrence
of bifurcation, as it happens with the constant and periodic solutions in the
autonomous case. The use of the skew-product formalism in the analysis
of the solutions of nonautonomous differential equations,
highly developed during the last decades,
has given several possible answers to this question: those objects
\textcolor{black}{we referred to}
can be bounded solutions, recurrent solutions, minimal sets,
local attractors, or global attractors.
The choice of each one of these categories defines a different
research line, and all of them have shown their relevance.
The works of Braaksma {\em et al.\/} \cite{brbh}, Alonso and Obaya \cite{alob3},
Johnson and Mantellini \cite{joma}, Fabbri {\em et al.\/} \cite{fajm},
Langa {\em et al.\/} \cite{lars1,lars2}, Rasmussen \cite{rasm2, rasm4},
N\'{u}\~{n}ez and Obaya \cite{nuob6}, P\"{o}tzsche \cite{potz2,potz4},
Anagnostopoulou and J\"{a}ger \cite{anja}, Fuhrmann~\cite{fuhr},
and Caraballo {\em et al.} \cite{clos},
develop some of these lines, providing nonautonomous
transcritical, saddle-node and pitchfork bifurcation patterns.
In some cases these nonautonomous phenomena admit a dynamical
description analogous to that of the autonomous case. But in other cases they
present some extremely complex dynamical phenomena,
which cannot occur in the autonomous scenery or even in the periodic one.
\par
\textcolor{black}{
The question of determining what a Hopf bifurcation means
in the nonautonomous case is even more complex. Among the
works devoted to this problem we can mention those of
Braaksma and Broer \cite{brbr} and Braaksma {\em et al.\/} \cite{brbh}
(who focus on the quasiperidic case and talk about bifurcation when
there is a change of dimension of invariant tori),
Johnson and Yi \cite{joyi} (where special attention is paid to the case in which an
invariant two-torus looses stability), Johnson {\em et al.\/} \cite{jokp}
(who study a two-step bifurcation pattern of Arnold type, in terms of the
global change of the local and pullback attractor)
Anagnostopoulou {\em et al.\/} \cite{anjk}
(where a Hopf bifurcation pattern for a discrete skew-product flow is analyzed,
and where the bifurcation corresponds to a global change in the global
attractor), and Franca {\em et al.\/} \cite{frjm}
(where conditions are established for nonautonomous perturbations
of a classical autonomous pattern of Andronov-Hopf bifurcation
ensuring the persistence of the perturbation),
as well as some of the references therein.
}
\par
The present work is the first of two papers devoted to the analysis
of Hopf bifurcation phenomena occurring for a one-parametric
family of families of nonautonomous two-dimensional systems of ODEs of the form
\begin{equation}\label{1.ecintro}
 \by'=A^\ep(\sigma(t,\w))\,\by-k_\rho(|\by|)\,\by\,,\qquad\w\in\W:
\end{equation}
each one of the systems is
defined along one of the orbits of a continuous flow $(\W,\sigma,\R)$ on
a compact metric space $\W$, which is minimal and uniquely ergodic.
Here, $A^\ep\colon\W\to\M_{2\times 2}(\R)$ is a family of continuous maps
for $\ep$ in a given interval;
$|\by|$ represents the Euclidean norm of $\by\in\R^2$; and the map
$k_\rho$ is given by
\[
 k_\rho\colon\R^+\!\to\R^+\!\,,\quad r\mapsto\left\{\begin{array}{ll}
 0&\text{if}\;\,0\le r\le\rho\,,\\
 (r-\rho)^2&\text{if}\;\,r\ge\rho\,,\end{array}\right.
\]
for a fixed value of $\rho\in(0,1]$.
\textcolor{black}{That is, $k$ is $C^1$ on
$\R^+$, convex, increasing and unbounded, and it vanishes on $[0,\rho]$.}
\par
Frequently, this setting comes
from a one-parametric system of the form
\[
 \by'=A_0^\ep(t)\,\by-k_\rho(|\by|)\,\by\,.
\]
\textcolor{black}{For instance, let $A_0$ be an almost-periodic matrix-valued
function on $\R$, and $A_0^\ep=A_0+\ep I_2$. In this case we can take
$\W$ as the hull $A_0$ (that is, the closure in the compact-open topology
of the set $\{A_s\,|\;s\in\R\}$ with $A_s(t)=A_0(t+s)$), define $\sigma$ by
time translation (that is, $\sigma(t,\w)(s)=\w(t+s)$), and define
$A(\w)=\w(0)$ (so that $A(\sigma(t,\w))=\w(t)$) and $A^\ep=A+\ep I_2$. The properties of
compactness of $\W$, of continuity, minimality and ergodic uniqueness
of the flow $\sigma$, and of continuity of the map $A$ are proved in
Sell~\cite{sell}. Note that the almost-periodic situation includes as particular
cases those in which $A_0$ is constant (with $\W$ given by a point) or periodic
(with $\W$ given by a circle), which are the simplest ones, as
well as those in which $A_0$ is quasiperiodic (with $\W$ given by a torus)
or limit-periodic (in which case $\W$ can be a solenoid: see \cite{mese},
and observe that in this example $\W$ is not a locally connected space,
so that it cannot be identified with a differentiable manifold).}
\par
\textcolor{black}{But we do not restrict ourselves to the almost-periodic
situation, which makes our setting more general; it is known
that there exist other functions for which the flow on the
hull is minimal and uniquely ergodic, although they are not easy to describe.
The main advantage of having a family like \eqref{1.ecintro}
instead of a single system}
is that the solutions of
all the systems corresponding to any fixed $\ep$
allow us to define a skew-product flow
$(\W\!\times\!\R^2\!,\tau^\ep_\R,\R)$
with $(\W,\sigma,\R)$ as base component.
\textcolor{black}{
In addition, the boundedness of $A^\ep$ combined with our particular
choice for $k_\rho$ ensures the dissipativity of the family for each value of $\ep$
(as we will check in Section \ref{sec4}), and hence the
existence of a global attractor $\mA^\ep\subset\W\!\times\!\R^2$
for $\tau_\R^\ep$.}
\par
The existence of this global attractor is the starting point:
our bifurcation analysis is focused on the evolution of
$\mA^\ep$ as $\ep$ varies, so that
a substantial change on its structure determines a bifurcation point.
We put special attention in
exploring the possibility of occurrence of Li-Yorke chaos at the bifurcation
points, which is indeed the situation in some examples that we describe.
\par
We will show that the occurrence of bifurcation points is
determined by the variation of the Sacker and Sell spectrum
$\Sigma_{A^\ep}$ of the systems $\by'=A^\ep(\sigma(t,\w))\,\by$, which
due to the minimality of the base flow can be defined for any fixed
$\w\in\W$: it is the compact set of points $\lambda\in\R$ such that
the translated system $\by'=(A^\ep(\sigma(t,\w))-\lambda I_2)\,\by$
does not have exponential dichotomy over $\R$.
Let $\lambda_-(\ep)$ and $\lambda_+(\ep)$ be the
left and right edge points of $\Sigma_{A^\ep}$. In this first paper we will analyze
the bifurcation occurring at the value $\wit\ep$
of the parameter when:
\textcolor{black}{
\begin{itemize}
\item[--] if $\ep<\wit\ep$, then $\lambda_+(\ep)<0$: every point in the spectrum is strictly negative;
\item[--] $\lambda_-(\wit\ep)=\lambda_+(\wit\ep)=0$: 0 is the unique point in the spectrum;
\item[--] if $\ep>\wit\ep$, then $\lambda_-(\ep)>0$: every point in the spectrum is strictly positive.
\end{itemize}}
\noindent
\textcolor{black}{
Note that in order to construct a family $A^\ep$ with these properties it suffices
to take as starting point a matrix $A$ with $\Sigma_{A}=\{0\}$
and then define $A^\ep:=A+\ep\,I_2$ (so that $\wit\ep=0$).
But many more situations may fit in our conditions.
We will show that the global attractor reduces to
$\W\times\{\bcero\}$ for $\ep<\wit\ep$. But, for $\ep>\wit\ep$, it
is given by a set which is homeomorphic to a solid cylinder around
$\W\times\{\bcero\}$; and in addition, the
boundary of the attractor is the global attractor for the flow
restricted to the set $\W\times(\R^2\!-\!\{\bcero\})$ (which is invariant, since
so is $\W\times\{\bcero\}$).}
\textcolor{black}{We will call this pattern
{\em nonautonomous Hopf bifurcation with zero spectrum}
due to the analogies of this structure with classical
autonomous Hopf bifurcation models. To understand these analogies,
just think about the classical autonomous model for
Hopf bifurcation $\by'=A^\ep\by-|\by|^2\by$
with $A^\ep=\lsm \;\,\,\ep&1\\-1&\ep\rsm$:
for $\ep<0$ the global attractor of the induced flow on $\R^2$
reduces to the origin of coordinates, while for $\ep>0$
it is given by a closed disk centered at the origin,
whose border attracts all the orbits
different from the origin.}
\par
\textcolor{black}{
At the bifurcation point $\wit\ep$, many possibilities arise. In the
simplest one, the attractor is again homeomorphic to a solid cylinder. Therefore
even in this case the bifurcation is discontinuous: just compare
with the behavior for $\ep<\wit\ep$.
And there are cases for which both the shape of the attractor and the dynamics
on it are extremely complex, with the occurrence of
Li-Yorke chaos.}
\par
Note that the specific form of $\mA^\ep$ is determined by the
Sacker and Sell spectrum of the linear part, and hence
by the matrix-valued function $A^\ep$.
The map $k_\rho$, responsible of the nonlinearity of the dynamics,
plays the role of guaranteing the dissipativity, and the
constant $\rho$ determines the global size of the attractor.
\textcolor{black}{
In addition, the fact that $k_\rho$ vanishes on $[0,\rho]$ is
fundamental to make the occurrence of Li-Yorke chaos possible,
and causes the bifurcation to be discontinuous even in the simplest
case.}
\par
\textcolor{black}{The second paper of the series will include the analysis of
a {\em nonautonomous two-step transcritical-Hopf bifurcation} pattern. In this case
we will assume the existence of $\wit\ep_1$ and $\wit\ep_2$ such that:
\begin{itemize}
\item[--] if $\ep<\wit\ep_1$, then $\lambda_+(\ep)<0$: every point in the spectrum is strictly negative;
\item[--] $\lambda_-(\wit\ep_1)<0=\lambda_+(\wit\ep_1)$: $0$ is the superior of the spectrum
but not its unique element;
\item[--] if $\ep\in(\wit\ep_1,\wit\ep_2)$, then $\lambda_-(\ep)<0<\lambda_+(\ep)$:
there are strictly negative and strictly positive points in the spectrum;
\item[--] $\lambda_-(\wit\ep_2)=0<\lambda_+(\wit\ep_2)$: $0$ is the inferior of the spectrum
but not its unique element;
\item[--] if $\ep>\wit\ep_2$, then $\lambda_-(\ep)>0$: every point in the spectrum is strictly positive.
\end{itemize}
The simplest example corresponding to this situation can be
$A^\ep:=\lsm \ep-1&0\\0&\ep\rsm$: here, $\wit\ep_1=0$ and $\wit\ep_2=1$.
For more complex (and always time-dependent) choices of $A^\ep$, we will show that
Li-Yorke chaos may appear at $\ep_1$ and/or $\ep_2$. And we will also describe the
possibility of persistence of the Li-Yorke chaos when
$\Sigma_{A^\ep}\subset(0,\infty)$. This last result is interesting for both patterns:
for the second one we have $\Sigma_{A^\ep}\subset(0,\infty)$ if $\ep>\wit\ep_2$; and
for the pattern analyzed in this paper, we have $\Sigma_{A^\ep}\subset(0,\infty)$
if $\ep$ is greater than the unique bifurcation point $\wit\ep$.}
\par
It is clear that
\textcolor{black}{the situations analyzed in these two papers}
are far away from exhausting the
possibilities. But they suffice to illustrate, once more, the extreme complexity
of the bifurcation phenomena in the nonautonomous case: there may appear
scenarios of dynamical unpredictability which
are not possible in the autonomous case.
\par
Let us sketch briefly the structure of the paper. In Section~\ref{sec2}
we recall the basic notions and results on topological dynamics and measure theory
which we will use. In the rest of this Introduction, $(\W,\sigma,\R)$
will always represent a continuous flow on a compact metric space, which
is assumed to be minimal and uniquely ergodic.
Also in Section \ref{sec2}, we pay special attention to the
skew-product flows induced on the bundles $\W\!\times\!\R^2$, $\WS$ and $\WP$
(where $\S$ is the unit circle in $\R^2$ and $\P$ is the real projective line)
by families of two-dimensional linear systems of ODEs
given by the evaluation of a continuous matrix along the orbits of the flow on $\W$.
\textcolor{black}{Systems of this type}
are also the object of analysis in
Section \ref{sec3}. We prove there that the flow given on $\WP$ by a
weakly elliptic family of linear systems is Li-Yorke chaotic in the case
that it admits an invariant measure which is absolutely continuous with
respect to the product measure on the bundle. Apart from the
intrinsic interest of this result, some of the properties shown in
its proof will be used in Section~\ref{sec6}.
\par From this point the paper is focused on the analysis of
the attractor $\mA^\ep\subset\W\!\times\!\R^2$ for a dissipative family of systems
of the type \eqref{1.ecintro}. Let us fix a value of the
parameter $\ep$, and omit the superscript on $A^\ep$ and $\mA^\ep$.
In Section \ref{sec4} we describe
this model in detail, as well as the flows induced on $\W\!\times\!\R^2$,
$\WS\!\times\!\R^+$ and $\WP\!\times\!\R^+$. We show that they are
dissipative, so that they admit global attractors, part of whose basic
properties we describe to complete the section.
\par
In Section \ref{sec5}
we relate the global shape and characteristics of the attractor $\mA$
with the characteristics of the Sacker and Sell spectrum $\Sigma_A$ of the family
of systems $\by'=A(\wt)\,\by$.
We do not contemplate in this paper all the possibilities:
since we are interested in the nonautonomous Hopf bifurcation
with zero spectrum pattern, the analysis will be
reduced to the cases $\sup\Sigma_A<0$, $\inf\Sigma_A>0$, and $\Sigma_A=\{0\}$.
As we have already mentioned, the attractor is trivial if
$\sup\Sigma_A<0$: $\mA=\W\times\{\bcero\}$. We have also mentioned that
$\mA$ takes the form of a solid cylinder around $\W$ with
continuous boundary if $\inf\Sigma_A>0$.
More precisely, for each $\w\in\W$ the section
$\mA_\w:=\{\by\in\R^2\!\,|\;(\w,\by)\in\mA\}$ contains and is homeomorphic to a
closed disk centered at $\bcero$, and the set
$\mA_\w$ varies continuously with respect to $\w$. And, in addition
the boundary of the \lq\lq cylinder", which is continuous and invariant,
is the attractor of the flow
restricted to the invariant set $\W\!\times\!(\R^2\!-\!\{\bcero\})$.
The attractor $\mA$ also takes the form of a solid cylinder with
continuous boundary in the case $\Sigma_A=\{0\}$
if in addition all the solutions of all the linear systems are bounded.
\textcolor{black}{
Therefore, even in this simplest case there is a lack of continuity
in the bifurcation.
We complete Section 5 by showing with some simple figures the evolution
of the global attractor when $A^\ep=A+\ep I_2$ and $A$ is a quasiperiodic
matrix-valued function fitting in the situation just
described, in order to clarify the sense of talking about
a Hopf bifurcation pattern.
}
\par
\textcolor{black}{
Finally, here we do not say too much about the general properties of $\mA$ if
$\Sigma_A=\{0\}$ and one unbounded solution exists.}
However, this last case is precisely the most interesting one for the purposes
of the paper.
The results obtained in Section \ref{sec3}
are a fundamental tool in Section \ref{sec6}, which
is devoted to establish conditions
ensuring the occurrence of Li-Yorke chaos,
in a very strong sense, on the attractor $\mA$.
We complete the section and the paper by showing that these
conditions are fulfilled in some interesting cases.
\textcolor{black}{For instance, when the family \eqref{1.ecintro}
is of the type
\[
 \by'=\big(\wit A(\sigma(t,\w))+(e(\sigma(t,\w))+\ep)\,I_2\big)\,\by-k_\rho(|\by|)\,\by\,,
 \qquad\w\in\W\,,
\]
where $\wit A$ has null trace, all the solutions
of all the linear systems of the family $\by'=\wit A(\wt)\,\by$
are bounded, and
$e\colon\W\to\R$ is a continuous function providing the following (highly complex)
dynamics for the flow induced on $\W\times\R$ by
the family of scalar equations $x'=e(\sigma(t,\w))\,x$: for almost every system
of the family (with respect to the unique ergodic measure) the
solutions are bounded; and there are systems for which the
solutions are not only unbounded but strongly oscillating
at $-\infty$ and $+\infty$. There are well known examples of
quasi-periodic functions $e_0\colon\R\to\R$ giving rise
to a hull $\W$ and a map $e$ with these characteristics,
as those described by Johnson in \cite{john9}
and Ortega and Tarallo in \cite{orta}. And recently
Campos {\em et al.} \cite{caot2} have proved that there
exist functions $e\colon\W\to\R$ with the required properties
whenever the (minimal and uniquely ergodic)
flow on $\W$ is not periodic.}
\par
\textcolor{black}{The conclusion is
that the carried-on analysis provides a pattern of nonautonomous Hopf bifurcation,
in which a extremely high degree of complexity is possible.
This possibility is one of the strongest differences with the
classical autonomous bifurcation theory.}
\section{Preliminaries}\label{sec2}
\subsection{Basic concepts}
We begin by recalling some basic concepts and properties of
topological dynamics and measure theory, and by fixing some notation.
\par
Let $\W$ be a complete metric space, and let $\dist_\W$ be
the distance on $\W$. A ({\em real and continuous})
{\em flow\/} on $\W$ is given by a continuous map
$\sigma\colon\R\!\times\!\W\to\W,\; (t,\w)\mapsto\sigma(t,\w)$
such that $\sigma_0=\text{Id}$ and $\sigma_{s+t}=\sigma_t\circ\sigma_s$
for each $s,t\in\R$, where $\sigma_t(\w):=\sigma(t,\w)$.
The flow is {\em local\/} if the map $\sigma$
is defined, continuous, and satisfies the previous properties on an open subset
of $\R\!\times\!\W$ containing $\{0\}\!\times\!\W$.
\par
Let the flow $(\W,\sigma,\R)$ be defined on
$\mU\subseteq\R\!\times\!\W$.
The set $\{\sigma_t(\w)\,|\;(t,\w)\in\mU\}$ is
the $\sigma$-{\em orbit\/} of the point $\w\in\W$. This orbit
is {\em globally defined\/} if $(t,\w)\in\mU$ for all $t\in\R$.
Restricting the time to $t\ge 0$ or
$t\le 0$ provides the definition of {\em forward\/}
or {\em backward $\sigma$-semiorbit}.
A subset $\mC\subseteq \W$ is {\em $\sigma$-invariant\/}
if it is composed by globally defined orbits; i.e., if
$\sigma_t(\mC):=\{\sigma(t,\w)\,|\;\w\in\mC\}$
is defined and agrees with $\mC$ for every $t\in\R$.
A $\sigma$-invariant subset $\mM\subseteq\W$ is {\em minimal\/} if it is compact
and does not contain properly any other compact $\sigma$-invariant set;
or, equivalently, if each one of the two semiorbits of anyone of
its elements is dense in it. The flow $(\W,\sigma,\R)$ is
{\em minimal\/} if $\W$ itself is minimal.
If the set $\{\sigma_t(\w_0)\,|\;t\ge 0\}$ is well-defined and
relatively compact, the {\em omega limit set\/} of $\w_0$ is given
by the points $\w\in\W$ such that $\w=\lim_{m\to \infty}\sigma_{t_m}(\w_0)$
for some sequence $(t_m)\uparrow \infty$.
This set is nonempty, compact, connected and $\sigma$-invariant.
By taking sequences $(t_m)\downarrow-\infty$ we obtain the definition of
the {\em alpha limit set\/} of $\w_0$. A global flow is {\em distal}
if $\inf_{t\in\R}\dist_\W(\sigma_t(\w_1),\sigma_t(\w_2))>0$ whenever $\w_1\ne\w_2$.
The next definitions are less standard:
\begin{defi}\label{2.defiLiYorke}
Let $(\W,\sigma,\R)$ be a continuous flow on a compact metric space.
Let $\w_1,\w_2$ be two points of $\W$ whose forward $\sigma$-orbits are
globally defined.
The points $\w_1,\w_2$ form a
{\em positively distal pair\/} for the flow
if $\liminf_{t\to\infty}\dist_{\W}(\sigma_t(\w_1),\sigma_t(\w_2))>0$, and
a {\em positively asymptotic pair\/} if
$\limsup_{t\to\infty}\dist_{\W}(\sigma_t(\w_1),\sigma_t(\w_2))=0$.
The points $\w_1,\w_2$ form
{\em Li-Yorke pair\/} for the flow if the pair is
neither positively distal nor positively asymptotic.
A set $\mS\subseteq\W$ such that every pair of different points of
$\mS$ form a Li-Yorke pair is called
a {\em scrambled set} for the flow. The flow $(\W,\sigma,\R)$
is {\em Li-Yorke chaotic}\/ if there exists an uncountable
scrambled set.
\end{defi}
This notion was introduced in \cite{liyo} in 1975.
The interested reader can find in \cite{bgkm} and \cite{akko}
some dynamical properties associated to Li-Yorke chaos and its relation with
other notions of chaotic dynamics.
\par
Let $m$ be a Borel measure on $\W$; i.e., a
regular measure defined on the Borel sets.
The measure $m$ is {\em $\sigma$-invariant\/}
if $m(\sigma_t(\mB))=m(\mB)$ for every Borel subset
$\mB\subseteq\W$ and every $t\in\R$. Suppose that $m$ is
finite and normalized; i.e., that $m(\W)=1$. Then it is
{\em $\sigma$-ergodic\/} if it is $\sigma$-invariant and,
in addition, $m(\mB)=0$ or $m(\mB)=1$ for every $\sigma$-invariant subset
$\mB\subseteq\W$.
If $\W$ is compact, the continuous flow $(\W,\sigma,\R)$
admits at least an ergodic measure. The flow is {\em uniquely
ergodic\/} if it admits just a normalized invariant measure, in which case
this measure is ergodic.
\par
Let $(\W,\sigma,\R)$ be a global flow on a compact metric space, and
let $\Y$ be a complete metric space. Let $\dist_\W$ and $\dist_\Y$ be
the distances on $\W$ and $\Y$. Then the map
$\dist_{\W\times\Y}((\w_1,y_1),(\w_2,y_2)):=
\dist_\W(\w_1,\w_2)+\dist_\Y(y_1,y_2)$ defines a distance on
$\W\!\times\!\Y$, and we have a new complete metric space.
In what follows, this
product space is understood as a bundle over $\W$.
The sets $\W$ and $\Y$ will be referred to as the
{\em base} and the {\em fiber} of the bundle.
A {\em skew-product flow
on $\W\!\times\!\Y$ projecting onto $(\W,\sigma,\R)$} is a (local or
global) flow given by a continuous map $\tau$ of the form
\[
 \tau\colon\mU\subseteq
 \R\!\times\!\W\!\times\!\Y\to\W\!\times\!\Y,\quad(\w,y)\mapsto
 (\wt,\tau_2(t,\w,y))\,.
\]
The flow $(\W,\sigma,\R)$ is the {\em base flow} of $(\W\!\times\!\Y,\tau,\R)$.
Note that the {\em fiber component\/} $\tau_2$ of $\tau$ satisfies
$\tau_2(s+t,\w,y)=\tau_2(s,\wt,\tau_2(t,\w,y))$ whenever the right-hand term is
defined. If $\Y$ is a vector space,
a global skew-product flow is {\em linear\/} if the map
$\Y\to\Y,\,y\mapsto \tau_2(t,\w,y)$ is
linear for all $(t,\w)\in\R\times\W$. A measurable map
$\alpha\colon\W\to\Y$ is an {\em equilibrium for $\tau$}
if $\tau_2(t,\w,\alpha(\w))=\alpha(\wt)$ for all $t\in\R$ and $\w\in\W$. A
set $\mK\subset\W\times\Y$ is a {\em copy of the base for $\tau$}
if it is the graph of a continuous equilibrium.
\begin{defi}
Let $(\W\!\times\!\Y,\tau,\R)$ be a skew-product flow
over a minimal and uniquely ergodic base $(\W,\sigma,\R)$, and let
$\mK\subseteq\W\!\times\!\Y$ be a $\tau$-invariant compact set. Then
the restricted flow $(\mK,\tau,\R)$ is {\em Li-Yorke fiber-chaotic in measure\/}
if there exists a set $\W_0\subseteq\W$
with full measure such that $\mK$ contains an uncountable scrambled set
of Li-Yorke pairs with first component $\w$ for each $\w\in\W_0$.
\end{defi}
\begin{nota}\label{2.notaLiYorke}
It is clear that, in the case of skew-product flow
$(\W\!\times\!\Y,\tau,\R)$, a pair of
points $(\w,y_1),\,(\w,y_2)$ (with common first component) form: a
positively distal pair if and only if $\liminf_{t\to\infty}\dist_\Y(\tau_2(t,\w,y_1),
\tau_2(t,\w,y_2))>0$; a positively asymptotic pair
if and only if $\limsup_{t\to\infty}\dist_\Y(\tau_2(t,\w,y_1),
\tau_2(t,\w,y_2))=0$; and a Li-Yorke pair if these two conditions fail.
\par
Note also that the notion of Li-Yorke fiber-chaos in measure, much more
exigent than that of Li-Yorke chaos, makes only sense in the setting of
skew-product flows. The same happens with the notion of residually
Li-Yorke chaotic flow, previously analyzed in \cite{bjjo}
and \cite{huyi}. Li-Yorke chaos for nonautonomous
dynamical systems is also the object of analysis in \cite{calo} and
\cite{clos}.
\end{nota}
The {\em Hausdorff semidistance\/}
from $\mC_1$ to $\mC_2$, where $\mC_1, \mC_2\subset\W\times\Y$, is
\[
 \dist(\mC_1,\mC_2):=\sup_{(\w_1,y_1)\in\mC_1}\left(\inf_{(\w_2,y_2)\in\mC_2}
 \big(\dist_{\W\times\Y}((\w_1,y_1),(\w_2,y_2))\big)\right).
\]
\begin{defi}\label{2.defiatractor}
A set $\mB\subset\W\!\times\!\Y$ is said {\em to attract a set
$\mC\subseteq\W$ under $\tau$} if
$\tau_t(\mC)$ is defined for all $t\ge 0$ and,
in addition, $\lim_{t\to\infty}\dist(\tau_t(\mC),\mB)=0$. The flow
$\tau$ is {\em bounded
dissipative\/} if there exists a bounded set $\mB$ attracting all
the bounded subsets of $\W\times\Y$ under $\tau$.
And a set $\mA\subset\W\times\Y$
is a {\em global attractor\/} for $\tau$ if it is compact, $\tau$-invariant, and it
attracts every bounded subset of $\W\!\times\!\Y$ under $\tau$.
\end{defi}
As usual, given a subset $\mC\subseteq\W\times\Y$, we will
represent its sections over the base elements by $\mC_\w:=
\{y\in\Y\,|\;(\w,y)\in\mC\}$.
Finally, given a normalized Borel measure on $m_\W$ on $\W$
and a regular measure $m_\Y$ on $\Y$, we represent by
$m_\W\!\times m_\Y$ the product measure on $\W\!\times\!\Y$.
A measure $m$ on $\W\!\times\!\Y$ {\em projects onto $m_\W$} if
$m(\mB\times\Y)=m_\W(\mB)$ for any Borel set $\mB\subseteq\W$. If this
is the case and $m$ is $\tau$-invariant, then $m_\W$ is $\sigma$-invariant.
\subsection{The flows induced by a linear family}
As usual, we identify the unit circle $\S$ of $\R^2$ and the one-dimensional
real projective line $\P$ with the quotient spaces
$\R/(2\pi\Z)$ and $\R/(\pi\Z)$, respectively. In this way, the map
\[
 \mkp\colon\S\to\P\,,\quad \theta\mapsto\theta(\text{mod}\:\pi)
\]
defines a projection of $\S$ onto $\P$. Note that
$\S$ can be understood as a 2-cover of $\P$: if $\theta\in\P$,
then $\mkp^{-1}(\theta)=\{\theta,\theta+\pi\}\subset\S$.
The map $\mkp$ will be frequently used.
\par
Let $(\W,\sigma,\R)$ be a minimal flow on a compact metric space.
(The ergodic uniqueness is not required by now.)
Given four continuous functions $a,b,c,d\colon\W\to\R$,
we consider the family of nonautonomous
two-dimensional linear systems of ODEs
\begin{equation}\label{2.eqlinear}
\by'=\left[\begin{array}{cc}a(\wt)&b(\wt)\\c(\wt)&d(\wt)\end{array}\right]\by
\end{equation}
for $\w\in\W$, with $\by\in\R^2\!$. We call $A:=\lsm a&b\\c&d\rsm$.
We will use the notation \eqref{2.eqlinear}$_\w$ to refer
to the system of this family corresponding to the point $\w$, and will proceed
in an analogous way for the rest of the equations appearing in the paper.
And we represent by $\by_l(t,\w,\by_0)$
the (globally defined) solution of the system \eqref{2.eqlinear}$_\w$
with initial data $\by_l(0,\w,\by_0)=\by_0$: the subindex $l$ makes reference to
the linearity of the systems. Then the map
\[
 \tau_{l,\R}\colon\R\!\times\!\W\!\times\!\R^2\!\to\W\!\times\!\R^2\!\,,
 \quad(t,\w,\by_0)\mapsto(\wt,\by_l(t,\w,\by_0))
\]
defines a linear skew-product flow with base $(\W,\sigma,\R)$.
It is possible to write
\begin{equation}\label{2.relx}
 \by_l(t,\w,\by_0)=
 \left[\!\begin{array}{c} r_l(t,\wth,r_0)\sin(\wih\theta(t,\wth))\\
 r_l(t,\wth,r_0)\cos(\wih\theta(t,\wth))\end{array}\!\right]
 \quad\text{for }\;\by_0=\left[\!\begin{array}{c} r_0\sin\theta\\
 r_0\cos\theta\end{array}\!\right]:
\end{equation}
here $t\mapsto\wih\theta(t,\wth)$ is the solution of the equation
\begin{equation}\label{2.eqtheta}
 \theta'=f(\wt,\theta)
\end{equation}
for
\begin{equation}\label{2.deffz}
 f(\wth):=-c(\w)\sin^2\!\theta+b(\w)\cos^2\!\theta+(a(\w)-
 d(\w))\sin\theta\cos\theta
\end{equation}
with initial data $\wih\theta(0,\wth)=\theta$, which we understand as
an element of $\S$; and the map
$t\mapsto r_l(t,\wth,r_0)$ solves
\begin{equation}\label{2.eqrlinear}
 r'=r\,g(\wt,\wih\theta(t,\wth))
\end{equation}
with $r_l(0,\wth,r_0)=r_0$, for
\begin{equation}\label{2.defgr}
 g(\wth):=
 a(\w)\sin^2\!\theta+d(\w)\cos^2\!\theta+(b(\w)+c(\w))\sin\theta\cos\theta\,.
\end{equation}
Note that the map
\begin{equation}\label{2.defsigma2}
 \wih\sigma\colon\R\!\times\!\WS\to\WS\,,
 \quad(t,\wth)\mapsto(\wt,\wih\theta(t,\wth))
\end{equation}
defines a skew-product flow on $\WS$.
Note also that $f(\wth)=f(\wth+\pi)$ and
$g(\wth)=g(\wth+\pi)$: we can define them either on $\WP$ or on $\WS$.
Consequently,
\[
 \wih\theta(t,\wth+\pi)=\wih\theta(t,\wth)+\pi
 \quad\text{and}\quad r_l(t,\wth,r_0)=r_l(t,\wth+\pi,r_0)\,.
\]
In particular, we
can understand the solutions of \eqref{2.eqtheta}
as elements $\P$: let us write $\wit\theta(t,\wth)=
\mkp(\wih\theta(t,\wth))=\wih\theta(t,\wth)(\text{mod}\:\pi)$
for $(t,\wth)\in\R\!\times\!\WP$,
and note that
\begin{equation}\label{2.defsigma3}
 \wit\sigma\colon\R\!\times\!\WP\to\WP\,,
 \quad(t,\wth)\mapsto(\wt,\wit\theta(t,\wth))
\end{equation}
defines a global skew-product flow on $\WP$. We say that
$(\WS,\wih\sigma,\R)$ {\em projects onto\/} $(\WP,\wit\sigma,\R)$.
Note also that $\mkp(\wih\theta(t,\wth))=\wit\theta(t,\w,\mkp(\theta))$
for $(t,\wth)\in\R\!\times\!\WS$.
\par
Let $U(t,\w)$ be the fundamental matrix solution of \eqref{2.eqlinear}$_\w$
with $U(0,\w)=I_2$, so that $\by_l(\w,\w,\by_0)=U(t,\w)\,\by_0$.
For further purposes we recall that
\begin{equation}\label{2.reldetU}
\det U(t,\w_0)=\exp\Big(\int_0^t \tr A(\w_0\pu s)\,ds\Big)\,.
\end{equation}
\begin{defi}\label{2.defEDsubfibrados}
The family \eqref{2.eqlinear} has {\em exponential dichotomy
over $\W$\/} if there exist
constants $c\ge 1$ and $\gamma>0$ and a splitting $\W\!\times
\R^2\!=F^+\!\oplus F^-$ of the bundle into the Whitney sum of
two closed subbundles such that
\begin{itemize}
 \item[-] $F^+$ and $F^-$ are invariant under
 the flow $(\W\!\times\!\R^2\!,\tau_{l,\R},\R)\,$,
 \item[-] $|\,U(t,\w)\,\by_0|\le c\,e^{-\gamma t}|\by_0|\quad$ for
   every $t\ge 0$ and $(\w,\by_0)\in F^+$,
 \item[-] $|\,U(t,\w)\,\by_0|\le c\,e^{\gamma t}|\by_0|\quad\;\;\,$
   for every $t\le 0$ and $(\w,\by_0)\in F^-$.
\end{itemize}
\end{defi}
\begin{notas} \label{2.notaED}
1.~Since the base flow $(\W,\sigma,\R)$ is minimal, the exponential
dichotomy of the family \eqref{2.eqlinear} over $\W$ is equivalent
to the exponential dichotomy of any of its systems over $\R$:
see e.g.~Theorem 2 and Section 3 of \cite{sase2}.
\par
2.~The family \eqref{2.eqlinear} has
exponential dichotomy over $\W$ if and only if no one of its systems
has a nontrivial bounded solution: see e.g.~Theorem 1.61 of \cite{jonnf}.
\end{notas}
\begin{defi}\label{2.defSSS}
The {\em Sacker and Sell spectrum\/} or {\em dynamical spectrum\/}
of the linear family \eqref{2.eqlinear}
is the set $\Sigma_A$ of $\lambda\in\R$ such that the family
$\by'=(A(\wt)-\lambda I_2)\,\by$
does not have exponential dichotomy over $\W$.
\end{defi}
Now we assume also that
the base flow $(\W,\sigma,\R)$ is uniquely ergodic, and represent by
$m_\W$ the unique $\sigma$-invariant (ergodic) measure.
The next theorem summarizes part of the information provided by
the Oseledets theorem (see Section 2 of \cite{jops})
and the Sacker and Sell spectral theorem (see Theorem 6 of \cite{sase6}).
\begin{teor}\label{2.teorOsel}
One of the following situations holds.
\par
\hypertarget{C1}{{\sc Case 1.}} $\Sigma_A=[\gamma_1,\gamma_2]$
with $\gamma_1<\gamma_2$.
In this case there exists a $\sigma$\nbd-invariant subset
$\W_0\subset\W$ with $m_\W(\W_0)=1$ such that, for all $\w\in\W_0$,
\begin{itemize}
\item[\rm o.1)] there exist
two one-dimensional vector spaces $W^{\gamma_1}_\w$ and $W^{\gamma_2}_\w$ such that
$\lim_{|t|\to\infty}(1/t)\ln|U(t,\w)\,\by_0|=
\gamma_j$ for $\by_0\in W^{\gamma_j}_\w-\{\bcero\}$.
\item[\rm o.2)] $\R^2\!=W^{\gamma_1}_\w\oplus W^{\gamma_2}_\w$.
\end{itemize}
In particular, if $\w\in\W_0$ and $\by^j=\lsm y^j_1\\y^j_2\rsm\in
W^{\gamma_j}_\w-\{\bcero\}$,
and we call $\theta^{\gamma_j}(\w):=\tan^{-1\!} (y^j_1/y^j_2)\in\P$, then
$\theta^{\gamma_j}(\w)$ is uniquely determined for $j=1,2$. In addition,
\begin{itemize}
\item[\rm o.3)] the maps $\W_0\to\P\,,\;\w\mapsto\theta^{\gamma_j}(\w)$
are measurable and satisfy $\wit\theta(t,\wth^{\gamma_j}(\w))=
\theta^{\gamma_j}(\wt)$ for all $t\in\R$ and $\w\in\W_0$, and for $j=1,2$.
\item[\rm o.4)] $\lim_{|t|\to\infty}(1/t)\ln\det U(t,\w)=\gamma_1+\gamma_2$
for all $\w\in\W$.
\end{itemize}
Consequently, the measurable subbundles
\begin{equation}\label{2.defOsel}
 W^{\gamma_j}:=\{(\w,\by)\in\W\!\times\!\R^2\!\,|\;
 \w\in\W_0 \text{ and }\by\in W^{\gamma_j}_\w\}
\end{equation}
are $\tau_{l,\R}$-invariant for $j=1,2$.
\par
\hypertarget{C2}{{\sc Case 2.}} $\Sigma_A=\{\gamma_1\}\cup\{\gamma_2\}$. Then all
the assertions in \hyperlink{C1}{\sc Case 1} hold for $\W_0=\W$. In addition:
$W^{\gamma_1}$ and $W^{\gamma_2}$ are closed subbundles; $\theta^{\gamma_1}$ and
$\theta^{\gamma_2}$ are continuous maps;
$\W\!\times\!\R^2\!=W^{\gamma_1}\oplus W^{\gamma_2}$
as Whitney sum;
$\lim_{t\to-\infty}\dist_{\P}(\wit\theta(t,\wth),\theta^{\gamma_1}(\wt))=0$
if $\theta\ne\theta^{\gamma_2}(\w)$; and
$\lim_{t\to\infty}\dist_{\P}(\wit\theta(t,\wth),\theta^{\gamma_2}(\wt))=0$
if $\theta\ne\theta^{\gamma_1}(\w)$.
\par
\hypertarget{C3}{{\sc Case 3.}} $\Sigma_A=\{\gamma\}$.
In this case, for all $\w\in\W$, $\lim_{|t|\to\infty}(1/t)\ln|\,U(t,\w)\,\by_0|=
\gamma$ for $\by_0\in \R^2\!-\{\bcero\}$, and
$\lim_{|t|\to\infty}(1/t)\ln\det U(t,\w)=2\gamma$.
\end{teor}
\begin{defi}
In {\sc cases 1} and 2, the sets $W^{\gamma_j}$ defined by
\eqref{2.defOsel} are the {\em Oseledets subbundles}
of the family of linear systems \eqref{2.eqlinear}, and
the numbers $\gamma_1$ and $\gamma_2$ are its
{\em Lyapunov exponents\/}.
In {\sc case 3}, the value $\gamma$ is the unique {\em Lyapunov exponent}.
\end{defi}
\begin{nota}\label{2.notarelac}
In the case that $\Sigma_A\subset(-\infty,0)$, $F^+=\W\!\times\!\R^2$;
and, if $\Sigma_A\subset(0,\infty)$, then $F^-=\W\!\times\!\R^2$.
These assertions follow easily from the fact that $\Sigma_A$ contains
all the Lyapunov exponents of the family (see
e.g.~Theorem 2.3 of \cite{jops}) and from the casuistic described in
Theorem~\ref{2.teorOsel}.
\end{nota}
\section{Li-Yorke chaos for weakly elliptic linear systems }\label{sec3}
As explained in the Introduction, this section provides conditions
on a certain type of linear systems
which ensure the occurrence of Li-Yorke chaos for the corresponding projective
flow. This is done in Theorem \ref{3.teorcaos}.
As a matter of fact, we will show that for almost every point $\w\in\W$,
the scrambled set of points of the form $(\wth)$
for the flow $(\WP,\wit\sigma,\R)$ contains all the points
of $\{\w\}\!\times\!\P$ excepting at most one.
Apart from the independent interest of this result,
its proof includes some arguments which will be essential in the
proof of our main result, in Section \ref{sec6}.
\par
For the rest of the paper, $(\W,\sigma,\R)$ will be a minimal and
uniquely ergodic flow on a compact metric space.
The unique $\sigma$-ergodic measure on $\W$ will be denoted by $m_\W$; $l_{\R^n}\!$
will be the Lebesgue measure on $\R^n$; and $l_\S$ and $l_\P$ will denote
the normalized Lebesgue measures on $\S$ and $\P$.
\par
We consider a family of linear (Hamiltonian) systems with zero trace,
\begin{equation}\label{3.eclineartraza0}
 \by'=\wit A(\wt)\,\by
\end{equation}
for $\w\in\W$, where $\wit A=\lsm \wit a&\,b\\c&-\wit a\rsm$.
(The choice of the names for the coefficients is due~to the fact that
$b$ and $c$ will later agree with those of \eqref{2.eqlinear}.)
The angular equation is
\begin{equation}\label{3.eqtheta}
 \theta'=f(\wt,\theta)
\vspace{-.2cm}
\end{equation}
for
\begin{equation}\label{3.deffz}
  f(\wth):=-c(\w)\sin^2\!\theta+b(\w)\cos^2\!\theta+
 2\,\wit a(\w)\sin\theta\cos\theta\,.
\end{equation}
As in the previous section,
we will represent by $\wih\theta(t,\w,\theta)$ and $\wit\theta(t,\w,\theta)$
the solutions on $\S$ and $\P$ of the equation with initial datum $\theta$,
and by $(\WS,\wih\sigma,\R)$ and $(\WP,\wit\sigma,\R)$ the corresponding flows,
given by the expressions \eqref{2.defsigma2} and \eqref{2.defsigma3}.
\begin{defi}\label{3.defiweak}
The family \eqref{3.eclineartraza0}
{\em is in the weakly elliptic case\/} if its Sacker and Sell spectrum
is $\Sigma_{\wit A}=\{0\}$ and at least one of its systems has an unbounded solution.
\end{defi}
We will describe in the next remarks two already classical
procedures which will be used several times
in what follows, and fix the corresponding notation.
\begin{nota}\label{3.notacambio}
Let $C\colon\W\to\M_{2\times 2}(\R)$ be continuous, with $\det C\equiv 1$, and
such that $C'\!\colon\W\to\M_{2\times 2}(\R)$ with
$C'(\w):=(d/dt)\,C(\wt)|_{t=0}$ is a well-defined and continuous
function.
Let us consider the change of variables $(\w,\by)\mapsto(\w,\bz)$ given by
$\bz=C(\w)\,\by$. This change of variables takes the system
$\by'=\wit A(\wt)\,\by$ to $\bz'=\wit A^*\!(\wt)\,\bz$ with
$\wit A^*\!:=(C'\!+CA)\,C^{-1}$. It is easy to check that $\tr\wit A^*=0$.
In addition,
\begin{itemize}
\item[\bf p1.] $\by'=\wit A(\wt)\,\by$ is in the weakly
elliptic case if and only if $\bz'=\wit A^*\!(\wt)\,\bz$ is, as
immediately deduced from $\bz(t,\w,\bz_0)=C(\wt)\,\by(t,\w,C^{-1}(\w)\,\bz_0)$
together with the boundedness of $C$ and $C^{-1}$.
\item[\bf p2.]
The linear and continuous change of variables
induces a homeomorphism $\Phi\colon\WP\to\WP$
with $\Phi(\wth)=(\w,\phi_\w(\theta))$ and
$\wit\sigma^*(t,\Phi(\wth))=\Phi(\wit\sigma(t,\wth))$.
It follows from here that
two points $(\wth_1),(\wth_2)$ are
a Li-Yorke pair for $(\WP,\wit\sigma,\R)$
if and only if the points
$(\w,\phi_\w(\theta_1)),(\w,\phi_\w(\theta_1))$
are a Li-Yorke pair  for $(\WP,\wit\sigma^*,\R)$.
The same argument shows that the property is also true for  positively
distal pairs instead of Li-Yorke pairs.
\item[\bf p3.] If the flow $(\WP,\wit\sigma,\R)$ admits an invariant measure
$\mu$ which is absolutely continuous with respect to $m_\W\!\!\times\!l_\P$, then
so does the flow $(\WP,\wit\sigma^*\!,\R)$ defined from the family
$\bz'=\wit A^*\!(\wt)\,\bz$. To prove it, we use the fact that, for all $\w\in\W$,
the homeomorphism $\phi_\w\colon\P\to\P$
takes measures of $\P$ which are absolutely continuous with respect to
$l_\P$ to measures of the same type. (In turn, this property can be deduced
from the fact that $\bz\mapsto C(\w)\,\bz$ preserves the Lebesghe measure
in $\R^2$, since $\det C(\w)=1$.) This assertion combined with
Fubini's theorem ensures that the measure $\mu^*$ defined over the
Borel sets by
$\mu^*(\mB):=\mu(\Phi^{-1}(\mB))$ (which is $\wit\sigma^*$-invariant) is
absolutely continuous with respect to $m_\W\!\!\times\!l_\P$.
(The maps $\Phi$ and $\phi_\w$ are defined in {\bf p2}.)
\end{itemize}
\end{nota}
\begin{nota}\label{3.notaextension}
Let us consider the flow $(\WS,\wih\sigma,\R)$.
We take a $\wih\sigma$-minimal set $\mM\subseteq\WS$.
For each $(\wth^1)\in\mM$, we define $\wit a^1\!(\wth^1):=
\wit a(\w),\,b^1\!(\wth^1):=b(\w)$ and $c^1\!(\wth^1):=c(\w)$, and consider
the family of linear systems
\begin{equation}\label{3.ecext}
\by'=\left[\begin{array}{rr} \wit a^1\!(\wih\theta(t,\wth^1))&
b^1\!(\wih\theta(t,\wth^1))\\
c^1\!(\wih\theta(t,\wth^1))&
-\wit a^1\!(\wih\theta(t,\wth^1))\end{array}\right]\by
\end{equation}
for $(\wth^1)\in\mM$. Note that the angular equation
\begin{equation}\label{3.eqthetaext}
 \theta':=-c^1\!(\wih\theta(t,\wth^1))\sin^2\!\theta+
 b^1\!(\wih\theta(t,\wth^1))\cos^2\!\theta+
 2\,\wit a^1\!(\wih\theta(t,\wth^1))\sin\theta\cos\theta
\end{equation}
corresponding to $(\wth^1)$ agrees with \eqref{3.eqtheta}$_\w$. Therefore,
the two skew-product flows with base $\mM$ defined by the family \eqref{3.eqthetaext}
are
\begin{eqnarray}
&\wih\vartheta_{\mM}\colon\R\!\times\!\mM\!\times\S\to\mM\!\times\S\,,\quad
(t,\wth^1\!,\theta)\mapsto (\wih\theta(t,\wth^1),
\wih\theta(t,\wth))\,,\nonumber\\
&\wit\vartheta_{\mM}\colon\R\!\times\!\mM\!\times\P\to\mM\!\times\P\,,\quad
(t,\wth^1\!,\theta)\mapsto (\wih\theta(t,\wth^1),
\wit\theta(t,\wth))\,.\label{3.flujovarP}
\end{eqnarray}
We list some properties relating \eqref{3.eclineartraza0} to \eqref{3.ecext},
later required.
\begin{itemize}
\item[\bf p4.]
If the family \eqref{3.eclineartraza0} is in the weakly elliptic case,
then so \eqref{3.ecext} is: 0 is its unique Lyapunov exponent,
and there exists an unbounded solution.
\item[\bf p5.]
Let us take $(\wth^1)\in\mM$ and $\theta_1,\theta_2\in\P$.
Then the points $(\wth_1),(\wth_2)$ are a Li-Yorke pair for $(\WP,\wit\sigma,\R)$
if and only if the points $(\wth^1\!,\theta_1),(\wth^1\!,\theta_2)$ are
a Li-Yorke pair for $(\mM\!\times\P,\wit\vartheta_{\mM},\R)$.
This fact follows immediately from the fact that the fiber
component of $\wit\vartheta_{\mM}(t,\wth^1\!,\theta)$ agrees with that
of $\wit\sigma(t,\wth)$. And the property is also true for positively
distal pairs.
\item[\bf p6.]
Let us assume that $(\WP,\wit\sigma,\R)$
admits an invariant measure $m$ which is absolutely continuous with respect to
$m_\W\!\!\times\!l_\P$, and
let $m_{\mM}$ be a $\wih\sigma$-ergodic measure on $\mM$, which
projects onto $m_\W$. Let $q\in L^1(\WP,m_\W\!\!\times\!l_\P)$ be the
density function of $m$, and let $f$ be defined by \eqref{3.deffz}.
Proposition 2.2 of \cite{obpa} ensures that there exists a measurable function
$p\colon\WP\to\R^+$
with $p(\wth)=q(\wth)$ for $(m_\W\!\!\times\!l_\P)$-a.a.~$(\wth)\in\WP$
such that
\begin{equation}\label{3.ecfuncional}
 p(\wit\sigma(l,\wth))=p(\wth)\exp\Big(-\int_0^l \frac{\p f}{\p\theta}\,
 (\wit\sigma(s,\wth))\,ds\Big)
\end{equation}
for all $(\wth)\in\WP$ and $l\in\R$. Let us define
$p^1(\wth^1\!,\theta):=p(\wth)$ for all $(\wth^1\!,\theta)\in\mM\!\times\!\P$.
It is easy to check that the nonnegative function
$p^1$ belongs to $L^1(\mM\!\times\P,m_\mM\!\!\times\!l_\P)$
and that it satisfies the equation \eqref{3.ecfuncional} corresponding to
$(\mM\!\times\P,\wit\vartheta_\mM,\R)$ for all
$(\wth^1\!,\theta)\in\mM\!\times\!\P$ and $l\in\R$. A new application of
Proposition 2.2 of \cite{obpa} shows that
$(\mM\!\times\P,\wit\vartheta_{\mM},\R)$
admits an invariant measure which is absolutely continuous with respect to
$m_\mM\!\!\times\! l_\P$.
\end{itemize}
Note also that the set
$\mM^1\!:=\{(\wth^1\!,\theta^1)\,|\;
(\wth^1)\in\mM\}$ is a copy of the base
for the flow $(\mM\!\times\S,\wih\vartheta_{\mM},\R)$.
The last property is the main achievement of this procedure: the
existence of this copy of the base (which may not be the case for
\eqref{3.eclineartraza0}), will allow us
to define linear and continuous changes of variables taking
\eqref{3.ecext} to families of systems whose corresponding dynamics
are easier to describe; and from this description
we will be able to derive the required conclusions for the
initial family.
\end{nota}
\begin{teor}\label{3.teorcaos}
Let us assume that the family \eqref{3.eclineartraza0} is in the weakly elliptic case.
Let us assume also that the flow $(\WP,\wit\sigma,\R)$
admits an invariant measure which is absolutely continuous with respect to
$m_\W\!\!\times\!l_\P$.
Then there exists a $\sigma$-invariant subset $\W_0\subseteq\W$ with
$m_\W(\W_0)=1$ such that for every $\w\in\W_0$ there exists a subset
$\mP_\w\subseteq\P$ with the next properties: $\P-\mP_\w$ contains
at most one element; and
for every pair of different points $\theta_1,\theta_2\in\mP_\w$, the
points $(\wth_1), (\wth_2)$ form a Li-Yorke pair for
$(\WP,\wit\sigma,\R)$. Hence, the flow $(\WP,\wit\sigma,\R)$
is Li-Yorke fiber-chaotic in measure.
\end{teor}
\begin{proof}
The proof is carried-out in two steps: the first one contains auxiliary
results for the second one, which proves the statements.
\vspace{.1cm}\par
{\sc Step 1}. We will begin by assuming that the
family \eqref{3.eclineartraza0} is triangular,
\begin{equation}\label{3.ectrian1}
\by'=\left[\begin{array}{cc}\wit a(\wt)&0\\c(\wt)&-\wit a(\wt)\end{array}\right]\by\,.
\end{equation}
Later on we will assume that the flow $(\WP,\wit\sigma,\R)$ either
does not contain a positively distal pair,
or it has two minimal sets.
\par
The angular equation for \eqref{3.ectrian1} is
$\theta'=2\,\wit a(\wt)\sin\theta\cos\theta-c(\wt)\sin^2\theta$,
so that the compact set
$\{(\w,0)\}\subset\WP$ is $\wit\sigma$-invariant. Therefore
it concentrates a $\wit\sigma$-invariant measure, which together with the
assumed existence of a $\wit\sigma$-invariant measure absolutely
continuous with respect to $m_\W\!\!\times\! l_\P$ allows us to apply
Proposition 3.3 of \cite{noob1} in order to conclude that:
there exist a $\sigma$-invariant set $\W_1$
with $m_\W(\W_1)=1$ and  measurable functions
$m_a\colon\W_1\to\R$ and
$\varphi_0\colon\W_1\to\P$ (with $(d/dt)\,m_a(\wt)=-\wit a(\wt)\,m_a(\wt)$
and such that $t\mapsto\varphi_0(\wt)$
satisfies the angular equation, in both cases
for all $\w\in\W_1$)
such that $\WP$ decomposes into the disjoint
union of the measurable $\wit\sigma$-invariant sets
$\mS_\eta:=\{(\w,\varphi_\eta(\w))\,|\;\w\in\W_1\}$ for $\eta\in(-\infty,\infty]$,
where $\varphi_\eta(\w):=\text{arccot\,}(\eta\,m_a^2(\w)+\cot\varphi_0(\w))$.
(These sets are the {\em ergodic $1$-sheets}, using the
language of \cite{noob1}; see also \cite{furs}.)
Let $\mK\subseteq\W_1$ be a compact set with $m_\W(\mK)>0$
such that the restrictions of $m_a$ and $\varphi_0$ to $\mK$ are continuous.
Let $\W_0^*\subseteq\W$ be the set of points $\w$ for which there
exists a sequence $(t_n)\uparrow\infty$ with $\wt_n\in\mK$.
Birkhoff's ergodic theorem ensures that $m_\W(\W_0^*)=1$.
Now we fix $\w\in\W_0^*$ and choose a sequence $(t_n)\uparrow\infty$ such
that $\wt_n\in\mK$ for all $n\ge 0$. We take two different points
$\theta_1,\theta_2\in\P$ and write them as
$\theta_i=\varphi_{\eta_i}(\w)$, so that $\eta_1\ne\eta_2$. Then
$\dist_\P(\wit\theta(t_n,\w,\theta_1),\wit\theta(t_n,\w,\theta_2))
\ge \inf_{\wit\w\in\mK}\dist_\P(\varphi_{\eta_2}(\wit\w),\varphi_{\eta_2}(\wit\w))>0$
if $n\ge 0$. In particular, for $\w\in\W_0^*$ and $\theta_1\ne\theta_2$,
\begin{equation}\label{3.paluego1}
\limsup_{t\to\infty}
\dist_\P(\wit\theta(t,\w,\theta_1),\wit\theta(t,\w,\theta_2))>0\,.
\end{equation}
\par
Now we consider two different situations. The first one is simple: the flow
$(\WP,\wit\sigma,\R)$ does not admit a positively distal pair.
Then \eqref{3.paluego1} ensures that the pair $(\w,\theta_1),\,(\w,\theta_2)$ is
of Li-Yorke type whenever $\w\in\W_0^*$ and $\theta_1,\theta_2\in\P$
with $\theta_1\ne\theta_2$. To complete the proof in this case,
we define $\W_0:=\W_0^*$ and
$\mP_{\w}:=\P$ for each $\w\in\W_0$.
\par
The second case we consider is that the flow
$(\WP,\wit\sigma,\R)$ admits two different minimal sets.
The proof of Proposition 4.4 in \cite{john6} (which
does not require the almost-periodicity of the base flow, there assumed)
shows that these minimal sets are the unique ones: otherwise all the
solutions of the systems of the family \eqref{3.ectrian1}
would be bounded, which is not the case.
We already know that one of
these minimal sets is $\{(\w,0)\,|\;\w\in\W\}$.
Therefore there exists $\delta>0$ such that any $(\wth^1)$
in the second minimal set satisfies $\delta\le\theta^1\le\pi-\delta$.
Let us take a $\wih\sigma$-minimal set $\mM\subset\WS$
projecting onto the second minimal set of $\WP$, and perform the procedure
described in Remark~\ref{3.notaextension}, taking now
\eqref{3.ectrian1} as starting point. Note that the
obtained system \eqref{3.ecext} is now also triangular: $b^1\equiv 0$.
In addition, the set $\mM^1\!:=\{(\wth^1\!,\theta^1)\,|\;
(\wth^1)\in\mM\}$ is a copy of the base for the flow
$(\mM\!\times\S,\wih\vartheta_{\mM},\R)$, and
it is contained either in $\WS\times[\delta,\pi-\delta]$
or in $[\pi+\delta,2\pi-\delta]$.
A straightforward computation shows that the
bounded and continuous change of variables on
$\mM\!\times\!\R^2$ given by $(\wth^1\!,\by)\mapsto(\wth^1\!,\bw)$ for
$\bw:=\lsm1&0\\-\cot\theta^1&1\rsm\by$
takes the (now triangular) family \eqref{3.ecext} to a new
family of the form
\begin{equation}\label{3.ecdiag1}
\bw'=\left[\begin{array}{cc} \wih a^{\,1}(\wih\theta(t,\wth^1))&
0\\0&-\wih a^{\,1}(\wih\theta(t,\wth^1))\end{array}\right]\bw\,.
\end{equation}
Let $m_\mM$ be a $\wih\sigma$-ergodic measure concentrated on $\mM$.
Property {\bf p1} in Remark~\ref{3.notacambio} combined with property {\bf p4} in
Remark~\ref{3.notaextension} ensures that the family \eqref{3.ecdiag1} is also
in the weakly elliptic case; and properties {\bf p3} and {\bf p6} ensure that the
new flow $(\mM\!\times\P,\wit\vartheta_\mM^*,\R)$ given by the
family of angular equations
$\theta'=2\,\wih a^{\,1}(\wih\theta(t,\wth^1))\sin\theta\cos\theta$
(corresponding to \eqref{3.ecdiag1}),
admits an invariant measure absolutely continuous with respect to
$m_\mM\!\!\times\! l_\P$. Clearly, the set
$\{(\wth^1\!,0)\,|\,(\wth^1)\in\mM\}
\subset\mM\!\times\!\P$ is $\wit\vartheta_\mM^*$-minimal
(as $\{(\wth^1\!,\pi/2)\,|\,(\wth^1)\in\mM\}$). Proposition 3.3 of \cite{noob1}
allows us to ensure that there exists
a $\wih\sigma$-invariant
subset $\mM_1$ of $\mM$ with $m_\mM(\mM_1)=1$ and
a measurable function $m_a^*\colon\mM_1\to\R^+$ with
$(d/dt)\,m_a^*(\wih\theta(t,\wth^1))=-\wih a(\wih\theta(t,\wth^1))\,
m_a^*(\wih\theta(t,\wth^1))$ for all $(\wth^1)\in\mM_1$
such that $\mM\!\times\!\P$
decomposes into the disjoint union of the measurable
$\wit\vartheta_\mM^*$-invariant sets (ergodic 1-sheets)
$\mS_\eta^*:=\{(\wth^1\!,\varphi_\eta^*(\wth^1))\,|\;(\wth^1)\in\mM_1\}$
for $\eta\in(-\infty,\infty]$,
where $\varphi_\eta^*(\wth^1):=\text{arccot\,}(\eta\,(m_a^*)^2(\wth^1))$.
\par
Let us call $\W_0$ to the intersection of the projection of $\mM_1$
onto $\W$ with the set $\W_0^*$ for which \eqref{3.paluego1} holds, and note that
there $\W_0$ is $\sigma$-invariant with $m_\W(\W_0)=1$. We fix $\w_0\in\W_0$ and
choose $\theta_0^1\in\S$ with $(\w_0,\theta_0^1)\in\mM_1$.
Given two different points $\theta_1,\theta_2\in\P$, we
write $\theta_i=\varphi_{\eta_i}^*(\w_0,\theta_0^1)$ for $i=1,2$,
so that $\eta_1,\eta_2\in(-\infty,\infty]$ and $\eta_1\ne\eta_2$.
Note that the fiber component of $\wit\vartheta_\mM^*$
satisfies $\wit\theta^*_\mM(t,\w_0,\theta_0^1,\theta_i)=
\varphi^*_{\eta_i}(\wih\sigma(t,\w_0,\theta_0^1))$.
Note also that there cannot exist $\kappa_1,\kappa_2\in\R$ such that
$0<\kappa_1\le m_a^*(\wih\sigma(t,\w_0,\theta_0^1))
\le\kappa_2<\infty$ for all $t\ge 0$: otherwise
all the solutions of all the systems of the family \eqref{3.ecext}
would be bounded (see e.g.~Proposition A.1 in \cite{jnot}), which is not the case.
Assume that there exists a sequence $(s_n)\uparrow\infty$
such that $\lim_{n\to\infty}m_a^*(\wih\sigma(s_n,\w_0,\theta_0^1))=0$.
Then $\lim_{n\to\infty}\wit\theta^*_\mM(s_n,\w_0,\theta_0^1,\theta_i)=\pi/2$
if $\theta_i\ne 0$ (since $\eta_i\ne\infty$).
In this case we set $\mP_{\w_0}^*:=\P-\{0\}$.
If the previous property does not hold, then
$\lim_{n\to\infty}m_a^*(\wih\sigma(s_n,\w_0,\theta_0^1))=\infty$
for a sequence $(s_n)\uparrow\infty$, which in turn ensures that
$\lim_{n\to\infty}\wit\theta^*_\mM(s_n,\w_0,\theta_0^1,\theta_i)=0$ if
$\theta_i\ne\pi/2$ (since $\eta_i\ne 0$).
In this case we take $\mP_{\w_0}^*:=\P-\{\pi/2\}$.
In both cases we conclude that $\liminf_{t\to\infty}
\dist_\P(\wit\theta^*_\mM(t,\w_0,\theta_0^1,\theta_1),
\wit\theta^*_\mM(t,\w_0,\theta_0^1,\theta_2))=0$ if
$\theta_1,\theta_2\in\mP_{\w_0}^*$.
\par
The performed change of variables
induces a flow homeomorphism from $(\mM\!\times\P,\wit\vartheta_\mM,\R)$ to
$(\mM\!\times\P,\wit\vartheta_\mM^*,\R)$, given by
$\Phi^*(\wth^1\!,\theta):=(\w,\text{arccot\,}(\cot\theta-\cot\theta^1))$.
Let us define $\mP_{\w_0}:=\{\theta\in\P\,|\;
\Phi^*(\wth^1\!,\theta)\in\mP_{\w_0}^*\}$. Then $\P-\mP_{\w_0}$ contains
one element; and, since the fiber component
of $\wit\vartheta_\mM$ is $\wit\theta$ (see \eqref{3.flujovarP}),
we have $\liminf_{t\to\infty}
\dist_\P(\wit\theta(t,\w_0,\theta_1),\wit\theta(t,\w_0,\theta_2))=0$ if
$\theta_1,\theta_2\in\mP_{\w_0}$.
Combining this property with \eqref{3.paluego1} we conclude that
for all $\w_0\in\W_0$ and all $\theta_1,\theta_2\in\mP_{\w_0}$
with $\theta_1\ne\theta_2$, the points $(\w_0,\theta_1), (\w_0,\theta_2)$
form a Li-Yorke pair for $(\WP,\wit\sigma,\R)$. This completes the proof in this case,
and {\sc step 1}. We point out that the ergodic uniqueness of $(\W,\sigma,\R)$
has not been required, which will be fundamental in {\sc step 2}.
\vspace{.1cm}\par
{\sc Step 2}. We will now prove the statement in the general case.
The idea is: to reformulate the family in a new base on which we can find
a continuous change of variables taking it to a new family with triangular form
for which the hypotheses remain true and which fits in one
of the two situations analyzed in {\sc step 1};
to apply the results of that step;
and to show that the conclusions for the new base suffice to our purposes.
\par
We follow again the procedure described in Remark~\ref{3.notaextension}:
we fix a $\wih\sigma$-minimal set $\mM\subseteq\WS$ and a
$\wih\sigma$-ergodic measure $m_\mM$ concentrated on $\mM$,
and consider the family of systems \eqref{3.ecext} and the flows
$(\mM\!\times\!\S,\wih\vartheta_\mM,\R)$ and
$(\mM\!\times\!\P,\wit\vartheta_\mM,\R)$.
(For the moment being $\mM$ is any set with these properties;
later on we will need to be more precise with its choice.)
Recall that
$\mM^1\!:=\{(\wth^1\!,\theta^1)\,|\;
(\wth^1)\in\mM\}$ is a copy of the base for $\wih\vartheta_\mM$.
Let us now consider the change of variables
given on $\mM\!\times\R^2$ by $(\wth^1\!,\by)\mapsto(\wth^1\!,\bz)$ with
$\bz:=\lsm\cos\theta^1\!&-\sin\theta^1\!\\
\sin\theta^1\!&\;\;\cos\theta^1\!\rsm\by$,
which induces the rotation $(\wth^1\!,\theta)\mapsto(\wth^1\!,
\theta-\theta^1)$ on $\mM\!\times\S$.
Clearly, the minimal set $\mM^1\!$ is taken to
$\{(\wth^1\!,0)\,|\;(\wth^1)\in\mM\}$, which ensures that the
solutions of \eqref{3.eclineartraza0} are taken to those of a new
family of the form
\begin{equation}\label{3.ectrian}
\bz'=\left[\begin{array}{cc} \wih a^{\,1}(\wih\theta(t,\wth^1))&0\\
\wih c^{\,1}(\wih\theta(t,\wth^1))&
-\wih a^{\,1}(\wih\theta(t,\wth^1))\end{array}\right]\bz:
\end{equation}
the coefficient $\wih b^1\!$ is $0$ since the function 0
solves the corresponding equation \eqref{3.eqthetaext}; and
the trace of the new matrix is zero, as explained in Remark~\ref{3.notacambio}.
\par
Let us now consider the
new flow $(\mM\!\times\P,\wit\vartheta_\mM^*,\R)$ given by the
family of angular equations
$\theta'=-\wih c^{\,1}(\wih\theta(t,\wth^1))\sin^2\!\theta
+2\,\wih a^{\,1}(\wih\theta(t,\wth^1))\sin\theta\cos\theta$
(corresponding to \eqref{3.ectrian}).
As in {\sc step 1}, we can ensure
that the family \eqref{3.ectrian} is in the weakly elliptic case,
and that the flow $(\mM\!\times\P,\wit\vartheta_\mM^*,\R)$
admits an invariant measure which is absolutely continuous with respect to
$m_\mM\!\!\times\! l_\P$. In other words, the family \eqref{3.ectrian}
satisfies the hypotheses of the theorem, and it is triangular.
We will see that a suitable choice if $\mM$ makes it fit in one
of the two situations analyzed in {\sc step 1}.
\par
Assume first that
the flow $(\WP,\wih\sigma,\R)$ does not admit a positively distal pair.
Then, according to the last assertion in {\bf p5} (in Remark~\ref{3.notaextension})
and {\bf p2} (in Remark~\ref{3.notacambio}),
the flow $(\mM\!\times\P,\wit\vartheta^*_{\mM},\R)$
does not admit a positively distal pair, and hence it fits in the first situation
of {\sc step 1} (no matter the choice of $\mM$).
\par
Now let us assume that the flow $(\WP,\wih\sigma,\R)$
admits a positively distal pair. We will check that
$\mM$ can be chosen in such a way that $(\mM\!\times\P,\wit\vartheta_{\mM},\R)$
admits two different minimal sets. Hence so does
$(\mM\!\times\P,\wit\vartheta_{\mM}^*,\R)$, and consequently
this flow fits in the second situation analyzed in {\sc step 1}.
\par
Let the points $(\wth_1),(\wth_2)$ form a positively distal
pair for $(\WP,\wit\sigma,\R)$. Let us take a $\wit\sigma$-minimal set
$\wit\mM\subseteq\WP$ contained in the omega limit set of $(\wth_1)$,
and a point $(\w_*,\theta^1_*)\in\wit\mM$. Then we can choose
a sequence $(t_n)\uparrow\infty$ and a point $(\w_*,\theta^2_*)\in\WP$
such that
$(\w_*,\theta^i_*)=\lim_{n\to\infty}\wit\theta(t_n,\wth_i)$
for $i=1,2$. Is is clear that
$(\w_*,\theta_*^1),(\w_*,\theta^2_*)$ form a positively distal pair.
Let us take a $\wih\sigma$-minimal set $\mM\subseteq\WS$
projecting onto $\wit\mM$ such that $(\w_*,\theta_*^1)\in\mM$,
consider the flow $(\mM\!\times\!\P,\wit\vartheta_{\mM},\R)$
defined by \eqref{3.flujovarP}, and
note that $(\w_*,\theta^1_*,\theta^1_*),
(\w_*,\theta^1_*,\theta^2_*)$ form a positively distal pair for
this flow. Note also that $(\w_*,\theta_*^1,\theta_*^1)$ belongs to the
minimal set $\mM^1\!:=\{(\wth^1\!,\mathfrak{p}(\theta^1))\,|\;
(\wth^1)\in\mM\}$, which is a copy of the base. Now we
take a minimal set $\mM^2$ contained in the omega limit set
of $(\w_*,\theta_*^1,\theta_*^2)$ for the flow $\wit\vartheta_{\mM}$.
It is easy to deduce from
$\inf_{t\ge 0}\dist_{\P}(\wit\theta(t,\w_*,\theta_*^1),
\wit\theta(t,\w_*,\theta_*^2))>0$
that $(\w_*,\theta_*^1,\theta_*^1)\notin\mM^2$, so that $\mM^1\!\ne\mM^2$.
This proves our assertion.
\par
In both cases, we have checked in {\sc step 1} that
there exists a $\wih\sigma$-invariant
subset $\mM_0\subseteq\mM$ with $m_\mM(\mM_0)=1$
such that for every $(\wth^1)\in\mM_0$
there exists a subset
$\mP_{(\wth^1)}\subseteq\P$ with the next properties: $\P-\mP_{(\wth^1)}$ contains
at most one element;
and for every pair of different points $\theta_1,\theta_2\in\mP_{(\wth^1)}$,
the points $(\wth^1\!,\theta_1), (\wth^1\!,\theta_2)$ form a Li-Yorke pair for
$(\mM\!\times\P,\wit\vartheta^*\!,\R)$. We define $\W_0$ as the projection of $\mM_0$
on $\W$ and note that it is $\sigma$-invariant. In addition, since
$m_\mM$ projects onto $m_\W$, we have $m_\W(\W_0)=1$.
Given $\w\in\W_0$ we look for $(\wth^1)\in\mM_0$ and define
$\mP_\w:=\mP_{(\wth^1)}$. Finally, we use the information provided by {\bf p2}
and {\bf p5} to conclude that $\W_0$ and $\{\mP_\w\,|\,\w\in\W_0\}$
satisfy all the assertions of the theorem. The proof is complete.
\end{proof}
\begin{nota}\label{3.notapaluego}
Let us consider the family \eqref{3.ectrian} appearing in {\sc step 2}
of the previous proof, and a $\wih\sigma$-ergodic measure
$m_\mM$ concentrated on $\mM$.
As seen at the beginning of the proof of {\sc step 1}, the
associated flow $(\mM\!\times\P,\wit\vartheta_\mM^*,\R)$ decomposes
into ergodic 1-sheets: we can write
$\mM\!\times\P=\bigcup_{\eta\in(-\infty,\infty]}\mS_\eta^*$,
where $\mS_\eta^*$ is a $\wit\vartheta^*\!$-invariant set
of the form $\{(\wth^1\!,\varphi_\eta^*(\wth^1))\,|
\;(\wth^1)\in\mM\}$ for a measurable map $\varphi_\eta^*\colon\mM\to\P$.
And, in addition, due to the expression of $\varphi_\eta^*$
(namely, $\varphi_\eta^*(\wth^1)=
\text{arccot\,}(\eta\,m_a^2(\wth^1)+\cot\varphi_0(\wth^1))$
for certain measurable functions $m_a$ and $\varphi_0$),
Lusin's theorem provides a compact set $\mK\subseteq\mM$
with measure $m_\mM(\mK)>1/2$ such that the restrictions of
all the maps $\varphi_\eta^*$ to $\mK$ are continuous.
On the other hand, the map $\mM\!\times\!\P\to\mM\!\times\!\P,\,
(\wth^1,\theta)\mapsto (\wth^1,\theta+\theta^1)$
takes $(\mM\!\times\P,\wit\vartheta_\mM^*,\R)$
to $(\mM\!\times\P,\wit\vartheta_\mM,\R)$. We define now
$\varphi_\eta(\wth^1):=\varphi_\eta^*(\wth^1)+\theta^1\!$
and note: that the restrictions of
all the maps $\varphi_\eta$ to $\mK$ are continuous; and that
the associated flow $(\mM\!\times\P,\wit\vartheta_\mM,\R)$ decomposes
into the ergodic $1$-sheets $\mS_\eta:=\{(\wth^1\!,\varphi_\eta(\wth^1))\,|
\;(\wth^1)\in\mM\}$ for $\eta\in(-\infty,\infty]$.
This information will be used in the proof of Theorem~\ref{6.teormain}.
\end{nota}
\section{The boundaries of the global attractors for the flows induced by a
family of dissipative systems}\label{sec4}
Recall that $(\W,\sigma,\R)$ is a continuous global flow
on a compact metric space, minimal and uniquely ergodic, that
$m_\W$ is its only ergodic measure, and that
$l_\R^2$, $l_\S$ and $l_\P$ represent the Lebesgue measures on
$\R^2$, $\S$ and $\P$ (normalized in the last two cases).
In the rest of the paper we will work with a particular type of
family of nonlinear systems defined
along the $\sigma$-orbits, which we now describe. Given
a real value $\rho\in(0,1]$ and the $C^1$-map
\begin{equation}\label{4.defkrho}
 k_\rho\colon\R^+\!\to\R^+\!\,,\quad r\mapsto\left\{\begin{array}{ll}
 0&\text{if}\;\,0\le r\le\rho\,,\\[.1cm]
 (r-\rho)^2&\text{if}\;\,r\ge\rho\,,\end{array}\right.
\end{equation}
we consider the family of nonautonomous two-dimensional systems of ODEs
\begin{equation}\label{4.eqR2}
\by'=A(\wt)\,\by-k_\rho(|\by|)\,\by
\end{equation}
for $\w\in\W$, whose linear part
\begin{equation}\label{4.eqlinear}
 \by'=A(\wt)\,\by
\end{equation}
agrees with \eqref{2.eqlinear}.
A discrete version of this continuous model has been
studied in \cite{anjk}. The main difference in our approach is
that we put the focus on the unpredictability of the dynamics
on the attractor at the bifurcation point.
\par
We will use the notation established in Section
\ref{sec2} for the flows $(\W\!\times\!\R^2\!,\tau_{l,\R},\R)$,
$(\WS,\wih\sigma,\R)$ and $(\WP,\wit\sigma,\R)$ induced by \eqref{4.eqlinear}.
The value of $\rho$ will be fixed throughout the paper, so that we will not include
it in the notation.
The family \eqref{4.eqR2} also induces a (now local) skew-product flow with base
$(\W,\sigma,\R)$ on $\W\!\times\!\R^2\!$, defined by
\begin{equation}\label{4.deftaur}
 \tau_\R\colon\mU\subseteq\R\!\times\!\W\!\times\!\R^2\!\to\W\!\times\!\R^2\!\,,
 \quad(t,\w,\by_0)\mapsto(\wt,\by(t,\w,\by_0))\,,
\end{equation}
where $\by(t,\w,\by_0)$
represents the solution of the system \eqref{4.eqR2}$_\w$
with initial data $\by(0,\w,\by_0)=\by_0$.
Note that $\by(t,\w,-\by_0)=-\by(t,\w,\by_0)$.
Note also that \eqref{4.eqR2} and \eqref{4.eqlinear} agree as long as
$\by$ belongs to the Euclidean closed disk centered at the origin
and with radius $\rho$; but not outside this disk, where
\eqref{4.eqR2} is no longer linear. In particular, $\by(t,\w,\by_0)$
may not be globally defined;
and if $|\by_0|\le\rho$ then $\by(t,\w,\by_0)=\by_l(t,\w,\by_0)$
at the interval of points containing 0 at which $|\by(t,\w,\by_0)|\le\rho$.
\par
Now we take coordinates $y_1=r\sin\theta$ and $y_2=r\cos\theta$
and obtain the equations
\begin{equation}\label{4.polar}
\begin{split}
\theta'&=f(\wt,\theta)\,,\\
r'&=r\,\big(g(\wt,\theta)-k_\rho(r)\big),
\end{split}
\end{equation}
where $f\colon\WS\to\R$ and $g\colon\WS\to\R$ are given by
\eqref{2.deffz} and \eqref{2.defgr}.
Recall that $f(\wth)=f(\wth+\pi)$ and
$g(\wth)=g(\wth+\pi)$.
Observe that the family of equations given by the first line in
\eqref{4.polar} depends neither
on $\rho$ nor on $r$. In fact, it coincides with the family
\eqref{2.eqtheta}, so that its solutions define the global flows $(\WS,\wih\sigma,\R)$
and $(\WP,\wit\sigma,\R)$ given by \eqref{2.defsigma2} and \eqref{2.defsigma3}.
Note also that the $r$ component of the solution of \eqref{4.polar}$_\w$ with
initial data $(\theta,r_0)$ agrees
with the solution of the equation
\begin{equation}\label{4.eqr}
 r'=r\,\big(g(\wt,\wih\theta(t,\wth))-k_\rho(r)\big)
\end{equation}
with initial data $r_0$. We can consider this family varying
either on $\WS$ or on $\WP$, since $g(\wt,\wih\theta(t,\wth))=
g(\wt,\wit\theta(t,\wth))$. Let $r(t,\wth,r_0)$ be the solution
of \eqref{4.eqr}$_{(\wth)}$ with $r(t,\wth,r_0)=r_0\ge 0$, and note that
if $r_0\le\rho$ then $r(t,\wth,r_0)=r_l(t,\wth,r_0)$ at  the
interval containing 0 at which $r(t,\wth,r_0)\le\rho$.
We can define two new local skew-product flows
with bases $(\WS,\wih\sigma,\R)$ and
$(\WP,\wit\sigma,\R)$, given by
\[
 \wih\tau\colon\wih\mV\subseteq\R\!\times\!\WS\!\times\!\R^+\!\to
 \WS\!\times\!\R^+,\quad
 (t,\wth,r_0)\mapsto
 (\wt,\wih\theta(t,\wth),r(t,\wth,r_0))
\]
and
\[
 \wit\tau\colon\wit\mV\subseteq\R\!\times\!\WP\!\times\!\R^+\!\to
 \WP\!\times\!\R^+,\quad
 (t,\wth,r_0)\mapsto
 (\wt,\wit\theta(t,\wth),r(t,\wth,r_0))\,.
\]
Recall that
\begin{equation}\label{4.SP}
 \mkp(\wih\theta(t,\wth))=\wit\theta(t,\w,\mkp(\theta))
 \quad\text{and}\quad
 r(t,\wth,r_0)=r(t,\w,\mkp(\theta),r_0)\,.
\end{equation}
It is clear that
\[
 \by(t,\w,\by_0)=
 \left[\!\begin{array}{c} r(t,\wth,r_0)\sin(\wih\theta(t,\wth))\\
 r(t,\wth,r_0)\cos(\wih\theta(t,\wth))\end{array}\!\right]
 \quad\text{if }\;\by_0=\left[\!\begin{array}{c} r_0\sin\theta\\
 r_0\cos\theta\end{array}\!\right].
\]
Therefore the flows $\tau_{\R}$ and $\wih\tau$ are closely related. As a matter
of fact, they can be identified outside the (respectively invariant) sets
$\W\!\times\!\{\bcero\}\subset\W\!\times\!\R^2$ and $\WS\!\times\{0\}\subset
\WS\!\times\!\R^+\!$. In addition, $(\WS\!\times\!\R^+\!,\wih\tau,\R)$ projects onto
$(\WP\!\times\!\R^+\!,\wit\tau,\R)$.
\begin{nota}\label{4.notaconcavo}
For further purposes, we point out that the skew-product semiflow
$\wit\tau$ is {\em concave}; that is, its fiber component satisfies
\[
 r(t,\wth,\eta\,r_1+(1-\eta)\,r_2)\ge \eta\,r(t,\wth,r_1)+(1-\eta)\,r(t,\wth,r_2)
\]
for all $\eta\in[0,1]$, $(\wth)\in\WP$, $r_1, r_2\in\R^+$, and all the
values of $t\ge 0$
such that all the involved terms are defined. This assertion follows from
the fact that $k_\rho(\eta r_1+(1-\eta)\,r_2)\le
\eta\,k_\rho(r_1)+(1-\eta)\,k_\rho(r_2)$
for all $\eta\in[0,1]$ and $r_1, r_2\in\R^+$ combined with a standard argument
of comparison of solutions. And, of course, it is also
{\em monotone}: $0\le r(t,\wth,r_1)\le r(t,\wth,r_2)$ whenever $0\le r_1\le r_2$.
The monotonicity will be often used without further reference.
\end{nota}
Since the function $g$ of \eqref{4.eqr} (given by \eqref{2.defgr}) is bounded,
the definition \eqref{4.defkrho} of $k_\rho$ shows that,
if we fix any $\delta>0$, we can find a $r_\rho$ such that
\begin{equation}\label{4.defrrho}
 g(\wth)-k_\rho(r)<-\delta\quad\text{ for
all $(\wth)\in\WS\;$ and $\,r\ge r_\rho$}\,.
\end{equation}
This constant $r_\rho$ plays a fundamental role in the statement of the
next result: it will bound the zones of the phase spaces in which the
attractors lie. Note also that the minimal choice of $r_\rho$
satisfying \eqref{4.defrrho} increases as $\rho$ increases. This is the only
point in which the choice of a particular $\rho\in(0,1]$ has influence.
\begin{teor}\label{4.teoratract}
Let the constant $r_\rho>0$ satisfy \eqref{4.defrrho}.
\begin{itemize}
\item[\rm(i)] The flow $(\W\!\times\!\R^2\!,\tau_\R,\R)$ is bounded dissipative,
and it has a global attractor $\mA\subset\W\!\times\!\R^2$
which contains $\W\!\times\!\{\bcero\}$. In addition,
\[
\begin{split}
 \mA&=\{(\w,\by_0)\in\W\!\times\!\R^2\!\,|\;
 \sup_{t\in\R}|\by(t,\w,\by_0)|<\infty\}\\
 &=\{(\w,\by_0)\in\W\!\times\!\R^2\!\,|\;
 \sup_{t\in\R^-}|\by(t,\w,\by_0)|<\infty\}\,;
\end{split}
\]
$\mA\subseteq\W\!\times\{\by\in\R^2\!\,|\;|\by|\le r_\rho\}$; and $(\w,\by_0)\in\mA$
if and only if $(\w,-\by_0)\in\mA$.
\item[\rm(ii)] The flow $(\WS\!\times\!\R^+\!,\wih\tau,\R)$ is bounded dissipative,
and it has a global attractor
$\wih\mB\subset\WS\!\times\!\R^+\!$
which contains $\WS\!\times\!\{0\}$. In addition,
\[
\begin{split}
 \wih\mB&=\{(\wth,r_0)\in\WS\!\times\!\R^+\!\,|\;
 \sup_{t\in\R}r(t,\wth,r_0)<\infty\}\\
 &=\{(\wth,r_0)\in\WS\!\times\!\R^+\!\,|\;
 \sup_{t\in\R^-}r(t,\wth,r_0)<\infty\}\,;
\end{split}
\]
$\wih\mB\subseteq\WS\times[0,r_\rho]$; and
$(\wth)\in\wih\mB$ for $\theta\in[0,\pi)$ if and only if
$(\wth+\pi)\in\wih\mB$.
\item[\rm(iii)] The set
\[
 \wit\mB:=\{(\w,\mkp(\theta),r_0)\,|\;(\wth,r_0)
 \in\wih\mB\}\subseteq\WP\!\times[0,r_\rho]
\]
is a global attractor for the bounded dissipative
flow $(\WP\!\times\!\R^+,\wit\tau,\R)$.
\item[\rm(iv)] If $r_0>0$, then
$\big(\w,\lsm r_0\sin\theta\\ r_0\cos\theta\rsm\big)\in\mA$ if and only if
$(\wth,r_0)\in\wih\mB$. In addition, the dynamics of
$\wih\tau$ on $\wih\mB-(\WS\!\times\!\{0\})$ can be recovered from that
of $\tau_\R$ on $\mA-(\W\!\times\!\{\bcero\})$, and the converse
is also true.
\end{itemize}
\end{teor}
\begin{proof}
It is easy to deduce from \eqref{4.defrrho} that
the maximal interval of definition of
$\by(t,\w,\by_0)$ contains $[0,\infty)$ for all $(\w,\by_0)\in\W\!\times\!\R^2\!$,
and that the set $\mC_\rho:=\W\times\{\by\in\R^2\,|\;|\by|\le r_\rho\}$
attracts any bounded set under $\tau_\R$. Or, equivalently,
that $r(t,\wth,r_0)$ is defined at least on $[0,\infty)$
for all $(\wth,r_0)\in\WS\!\times\R^+\!$ and that the set
$\wih\mC_\rho:=\WS\times[0,r_\rho]$ attracts any bounded set under $\wih\tau$.
Therefore both flows are bounded dissipative, and the classical theory
ensures that the sets $\mA$ and $\wih\mB$ of (i)
and (ii) are the respective global attractors: see e.g.~Section 2.4 of
\cite{hale4} and Section 1.2 of \cite{chlr}.
Relation \eqref{4.defrrho} also guarantees that
any globally bounded solution $\by(t,\w,\by_0)$ of
\eqref{4.eqR2} satisfies $|\by(t,\w,\by_0)|\le r_\rho$ for all $t\in\R$:
if $|\by_0|> r_\rho$, then \eqref{4.defrrho} would force
$|\by(t,\w,\by_0)|$ to tend to $\infty$ as $t$ tends to the left edge of the
maximal interval of definition.
so that it is not bounded; therefore $|\by_0|\le r_\rho$, and
hence we can use again \eqref{4.defrrho} to prove the assertion.
Consequently, $\mA\subseteq\mC_\rho$ and
$\wih\mB\subseteq\wih\mC_\rho$. And obviously
$\by(t,\w,\by_0)$ is globally bounded if and only if
$\by(t,\w,-\by_0)=-\by(t,\w,\by_0)$ is globally bounded, which completes
the proof of (i) and (ii). The assertions in (iii) follows from
(ii) and \eqref{4.SP}. Finally,
the properties sated in (iv) follow from (i) and (ii) and from the relation between
$\tau_\R$ and $\wih\tau$ explained before.
\end{proof}
\begin{teor}\label{4.teortapa}
Let $\mA\subset\W\!\times\!\R^2$, $\wih\mB\subset\WS\!\times\!\R^+\!$
and $\wit\mB\subset\WP\!\times\!\R^+\!$
be the global attractors of the flows $(\W\!\times\!\R^2\!,\tau_\R,\R)$,
$(\WS\!\times\R^+\!,\wih\tau,\R)$
and $(\WP\!\times\R^+\!,\wit\tau,\R)$, and let $r_\rho$
satisfy \eqref{4.defrrho}. Then, there exists an upper
semicontinuous map
\[
 \wit\beta\colon\WP\to[0,r_\rho]
\]
such that
\begin{itemize}
\item[\rm (i)] the attractors are given by
\[
\begin{split}
 \mA&=\bigcup_{\w\in\W}\{(\w,\by)\,|\;
 \by=\lsm r\sin\theta\\r\cos\theta\rsm\,
 \text{ for $\theta\in\S$ and $r\in[0,\wit\beta(\w,\mkp(\theta))]$}\}\,,\\
 \wih\mB&=\bigcup_{(\wth)\in\WS}\big(\{(\wth)\}\times
 [0,\wit\beta(\w,\mkp(\theta))]\big)\,,\\
 \wit\mB&=\bigcup_{(\wth)\in\WP}\big(\{(\wth)\}\times
 [0,\wit\beta(\wth))]\big)\,.
\end{split}
\]
\item[\rm(ii)] $\wit\beta(\wit\sigma(t,\wth))=r(t,\wth,\wit\beta(\wth))$
for all $(t,\wth)\in\R\!\times\!\WP$; that is, $\wit\beta$ is an equilibrium for
$\wit\sigma$.
\item[\rm(iii)] If there exists $\beta_0>0$
such that $\wit\beta(\wth)\ge\beta_0$ for all $(\wth)\in\WP$,
then $\wit\beta$ defines a uniformly stable equilibrium: for each
$\ep>0$ there exists $\delta(\ep)>0$ such that if $(\wth)\in\WP$ and
$0\le|\wit\beta(\wth)-r_0|\le\delta(\ep)$ then
\begin{equation}\label{4.r}
 |\wit\beta(\wit\sigma(t,\wth))-r(t,\wth,r_0)|\le\ep \quad
 \text{for all $t\ge 0$}\,.
\end{equation}
\end{itemize}
\end{teor}
\begin{proof}
Recall that the flow $(\WP\!\times\!\R^+,\wit\tau,\R)$ is induced by the
family of equations $r'=r\,\big(g(\wt,\wit\theta(t,\wth))-k_\rho(r)\big)$,
which agree with \eqref{4.eqr}. Relation \eqref{4.defrrho} ensures that
$0>r_\rho\,\big(g(\wih\theta(t,\wth))-k_\rho(r_\rho)\big)$, so that
the function $r(t)\equiv r_\rho$ is an upper solution for all these equations.
In addition, the set $\{r(t,\wth,r_\rho)\,|\;t\ge 0,\,(\wth)\in\WP\}$
is bounded, since $(\WP\!\times\!\R^+\!,\wit\tau,\R)$ is bounded dissipative.
In these conditions, Theorem 3.6 of \cite{nono2} (see also its proof,
which does not require the minimality of the base flow, in our case
$(\WP,\wit\sigma,\R)$) shows that
\[
 \wit\beta(\wth):=\lim_{t\to\infty}r(t,\wit\sigma(-t,\wth),r_\rho)
\]
defines an upper semicontinuous function satisfying $\wit\beta(\wth)\le r_\rho$
and property (ii).
In addition, if $r(t,\w_0,\theta_0,r_0)$ is globally bounded,
then Theorem \ref{4.teoratract}(i) ensures that
$r(-t,\w_0,\theta_0,r_0)\le r_\rho$ for all $t\ge 0$, so that
$r_0=r(t,\wit\sigma(-t,\w_0,\theta_0),r(-t,\w_0,\theta_0,r_0))
\le r(t,\wit\sigma(-t,\w_0,\theta_0),r_\rho)$
and hence $r_0\le\wit\beta(\w_0,\theta_0)$.
And conversely, since (ii) holds, any solution $r(t,\w_0,\theta_0,r_0)$
with $r_0\le\wit\beta(\w_0,\theta_0)$
is globally bounded. This and the descriptions of $\mA$, $\wih\mB$ and
$\wit\mB$ made in Theorem~\ref{4.teoratract} complete the proof of (i) and (ii).
\par
Let us prove (iii). Since $k_\rho\!(\eta\,r)$ is smaller than
$\eta\,k_\rho\!(r)$ for $r\ge 0$ and $\eta\in[0,1]$ and greater for
$r\ge 0$ and $\eta\ge 1$, a standard argument of comparison of
solutions and property (ii) ensure that
\[
\begin{split}
 \eta\,\wit\beta(\wit\sigma(t,\wth))\le r(t,\w,\eta\,\wit\beta(\wth)) \quad&
 \text{for all $t\ge 0$ and $(\wth)\in\WP$ if $\eta\in[0,1]$}\,,\\
 \eta\,\wit\beta(\wit\sigma(t,\wth))\ge r(t,\w,\eta\,\wit\beta(\wth)) \quad&
 \text{for all $t\ge 0$ and $(\wth)\in\WP$ if $\eta\ge 1$}\,.
\end{split}
\]
Given $\ep>0$ we take $\delta(\ep)=\ep\,\beta_0/r_\rho$, where $\beta_0$ satisfies
the assumption in (iii). Then,
if $|\wit\beta(\wth)-r_0|\le\delta(\ep)$ for a point $(\wth)\in\WP$,
we have $(1-\ep/r_\rho)\,\wit\beta(\wth)\le r_0\le
(1+\ep/r_\rho)\,\wit\beta(\wth)$. Combining this fact, the previous properties,
and (ii),
we get $(1-\ep/r_\rho)\,\wit\beta(\wit\sigma(t,\wth))\le r(t,\wth,r_0)
\le (1+\ep/r_\rho)\,\wit\beta(\wit\sigma(t,\wth))$ for all $t\ge 0$,
which together with $0\le\wit\beta(\wit\sigma(t,\wth))\le r_\rho$ yields \eqref{4.r}.
\end{proof}
It is clear that the properties of the semicontinuous map $\wit\beta$
of Theorem \ref{4.teortapa} determine the shapes of the three global attractors.
Let us see that $\wit\beta(\wth)>0$ if and only if the solutions $r_l(t,\wth,r_0)$ of
\eqref{2.eqrlinear}$_{(\wth)}$ with initial data $r_0>0$ are bounded. By linearity,
it is enough to consider the solutions $r_l(t,\wth,1)$.
\begin{prop}\label{4.propacotadas}
Let us fix $(\wth)\in\WP$, and let $\wit\beta$ be defined in
Theorem \ref{4.teortapa}.
Then $\wit\beta(\wth)>0$ if and only if $\sup_{t\le 0}r_l(t,\wth,1)<\infty$.
\end{prop}
\begin{proof}
Recall that $\rho\in(0,1]$ is fixed from the beginning, and that
$\wit\beta$ takes values in $[0,r_\rho]$ with $r_\rho$ satisfying \eqref{4.defrrho}.
We fix $(\wth)\in\WP$ and assume that $\beta_0:=\wit\beta(\wth)>0$.
Since $k_\rho\ge 0$, a standard arguments of comparison of
solutions for scalar ODEs applied to \eqref{2.eqrlinear} and
\eqref{4.eqr}, and Theorem~\ref{4.teortapa}(ii), guarantee
\[
 \beta_0\,r_l(t,\wth,1)=
 r_l(t,\wth,\beta_0)\le r(t,\wth,\beta_0)=\wit\beta(\wit\sigma(t,\wth))\le r_\rho
 \quad\text{for all $t\le 0$}\,.
\]
Therefore $\sup_{t\le 0}r_l(t,\wth,1)<\infty$, as
asserted.
\par
Assume now that $\kappa:=\sup_{t\le 0}r_l(t,\wth,1)<\infty$,
so that $\sup_{t\le 0}r_l(t,\wth,\rho/\kappa)\le \rho$.
Hence $r(t,\wth,\rho/\kappa)=r_l(t,\wth,\rho/\kappa)\le \kappa$ for all $t\le 0$,
which together with the dissipativity of the flows ensures that
$r(t,\wth,\rho/\kappa)$ is globally defined and bounded.
Theorem \ref{4.teoratract}(ii)\&(iii) ensure that
$(\wth,\rho/\kappa)\in\wit\mB$, and hence Theorem~\ref{4.teortapa}(ii)
guarantees that $\wit\beta(\wth)\ge\rho/\kappa$.
This completes the proof.
\end{proof}
The previous result shows that the sets
\[
\begin{split}
(\WP)^+\!&:=\{(\wth)\in\WP\,|\;\wit\beta(\wth)>0\}\,,\\
(\WP)^0&:=\{(\wth)\in\WP\,|\;\wit\beta(\wth)=0\}
\end{split}
\]
can be characterized in terms of the existence of bounded solutions on $\R^-$ for
the family of equations~\eqref{2.eqrlinear}; or, equivalently, of the
family of systems \eqref{4.eqlinear}. Note also that
both sets are $\wit\sigma$-invariant, as Theorem~\ref{4.teortapa}(ii) guarantees.
\begin{notas}\label{4.notaopciones}
1.~Since two different points in $\P$ determine two
linearly independent solutions of \eqref{4.eqlinear}$_\w$, Proposition
\ref{4.propacotadas} shows that there are three possibilities for the sections
$(\WP)^+_\w$ and $(\WP)^0_\w$:
$(\WP)^+_\w\!=\P$ and $(\WP)^0_\w$ is empty;
$(\WP)^+_\w$ is empty and $(\WP)^0_\w=\P$; and
$(\WP)^+_\w\!=\{\theta_0\}$ and $(\WP)^0_\w=\P-\{\theta_0\}$.
In addition, if a point $\w$ is in one of these cases, the same
happens with all the points $\wt$ of its $\sigma$-orbit: for all
$\theta_0\in\P$ and all $t\in\R$ there exists a unique $\theta_{-t}\in\P$ with
$(\wt,\theta_0)=\wit\sigma(t,\wth_{-t})$, namely
$\theta_{-t}:=\wit\theta(-t,\wt,\theta_0)$.
\par
2. The upper semicontinuity of $\wit\beta$ ensures that it is continuous at
every point of a residual set $\mC\subseteq\WP$. In addition, since
$\wit\beta\ge 0$, $(\WP)^0\subseteq\mC$.
Assume now that $(\WP)^0$ contains a set $\mD$ which is dense in $\WP$.
It is easy to deduce that $\wit\beta$ vanishes at any point of $\mC$.
So that, in this case, the set $(\WP)^0$ is residual in $\WP$.
\end{notas}
We also define the set
\begin{equation}\label{4.defOmega+}
 \W^+\!:=\{\w\in\W\,|\;(\WP)^+_\w\!=\P\}=\{\w\in\W\,|\;\wit\beta(\wth)>0
 \text{ for all $\theta\in\P$}\}\,,
\end{equation}
which is $\wit\sigma$-invariant: see Remark~\ref{4.notaopciones}.1.
Proposition~\ref{4.propacotadas} shows that
$\W^+\!=\W$ if and only if all the solutions of
all the systems \eqref{4.eqlinear}
are bounded on $\R^-$. Let us check that, if this is not the case,
then $\W^+\!$ is of the first Baire category.
\begin{prop}\label{4.propBaire}
Suppose that $\W^+\!\ne\W$. Then $\wit\beta$ vanishes exactly at the
residual $\wit\sigma$-invariant set of
points of $\WP$ at which it is continuous.
In particular, the set $\W^+\!$ is of the first Baire category.
\end{prop}
\begin{proof}
We will check below that $\wit\beta$ vanishes at
a dense set of points of $\WP$, which according to
Remark \ref{4.notaopciones}.2 ensures that the $\wit\sigma$-invariant set
$(\WP)^0$ is the residual set of continuity points of $\wit\beta$.
Since
$(\WP)^0$ projects onto $\W-\W^+\!$, this set is also residual
(see Proposition 3.1 of \cite{veec4}), which proves the second
assertion.
\par
Let us take $\w_0\notin\W^+\!$ and note that
$(\WP)^0_{\w_0}\supseteq\P-\{\theta_0\}$ for a point $\theta_0\in\P$
(see Remark~\ref{4.notaopciones}.1).
It follows easily from the minimality of $(\W,\sigma,\R)$ that the set
$\mD:=\{\wit\sigma(t,\w_0,\theta)\,|\;t\in\R\,,\theta\ne\theta_0\}$ is dense in $\WP$.
Finally, the $\wit\sigma$-invariance of $(\WP)^0$ ensures that
$\wit\beta$ vanishes at the points of $\mD$.
\end{proof}
We finish this section by relating the measure of $\W^+\!$
to those of the attractors.
\begin{prop}\label{4.propmedidaW+}
The attractor $\wit\mB$ of the flow $(\WP\!\times\!\R^+\!,\wit\tau,\R)$
has positive measure $m_\W\!\!\times\!l_\P\!\times\!l_\R$ on $\WP\!\times\!\R^+\!$
if and only if $m_\W(\W^+\!)=1$, where $\W^+\!$ is the $\sigma$-invariant set
defined by \eqref{4.defOmega+}.
\end{prop}
\begin{proof}
It is clear that the measure of $\wit\mB$ is given by
$\int_{\WP}\wit\beta(\wth)\,d(m_\W\!\!\times\!l_\P)$, which agrees with
$\int_{(\WP)^+\!}\wit\beta(\wth)\,d(m_\W\!\!\times\!l_\P)$;
so that this measure is positive if and only if $(m_\W\!\!\times\!l_\P)((\WP)^+\!)>0$.
In turn, $(m_\W\!\!\times\!l_\P)((\WP)^+\!)=\int_\W l_\P((\WP)^+_\w\!)\,dm_\W$,
so that it is positive
if and only if the set of points $\w$ with $l_\P((\WP)^+_\w\!)>0$
has positive measure $m_\W$. It follows from Remark~\ref{4.notaopciones}.1 that
this set agrees with $\W^+\!$. And, since $\W^+\!$ is $\sigma$-invariant and
$m_\W$ is $\sigma$-ergodic, $m_\W(\W^+\!)>0$ is
equivalent to $m_\W(\W^+\!)=1$.
\end{proof}
\begin{nota}\label{4.notamedida}
Note that $(m_\W\!\!\times\!l_\P\!\times\!l_\R)(\wit\mB)>0$ is equivalent to
$(m_\W\!\!\times\!l_\S\!\times\!l_\R)(\wih\mB)>0$ and to
$(m_\W\!\!\times\!l_{\R^2\!})(\mA)>0$: see Theorems~\ref{4.teoratract}
and \ref{4.teortapa}.
\end{nota}
\section{The shape of the global attractors in terms of the
Sacker and Sell spectrum of the associated linear system}\label{sec5}
We continue in this section with the analysis of the shape and properties of
the global attractor $\mA$ (and hence of $\wih\mB$ and $\wit\mB$)
associated to the family of systems \eqref{4.eqR2} described
at the beginning of Section \ref{sec4}. As explained in the
Introduction, we will relate these properties to the characteristics
of the Sacker and Sell spectrum $\Sigma_A$ of the linear family \eqref{4.eqlinear}
(or of any one of its systems: see Definitions \ref{2.defEDsubfibrados} and
\ref{2.defSSS}, and Remark~\ref{2.notaED}.1).
We have also explained there that in this paper we will consider just three cases:
$\Sigma_A\subset(-\infty,0)$, $\Sigma_A\subset(0,\infty)$, and
$\Sigma_A=\{0\}$, and that this casuistic can be understood as a bifurcation
pattern.
\par
Recall that the boundaries of the attractors are determined by
the function $\wit\beta\colon\WP\to[0,r_\rho]$ of Theorem \ref{4.teortapa}.
\vspace{.2cm}\par\noindent
{\bf \hypertarget{A}{A}. The case $\boldsymbol{\Sigma_A\subset(-{\infty},0)}$.}
This first case is the simplest one.
We will check that the attractors $\mA$, $\wih\mB$ and $\wit\mB$ are trivial.
Note that the conclusion is irrespective
of the fact that $\Sigma_A$ reduces to a point, it is a nondegenerate interval,
or it is composed by two negative points, which are the three possibilities:
see Theorem~\ref{2.teorOsel}.
\begin{teor}\label{4.teorattrzero}
Suppose that $\Sigma_A\subset(-\infty,0)$. Then
$\mA=\{(\w,\bcero)\,|\;\w\in\W\}$, $\wih\mB=
\{(\wth,0)\,|\;(\wth)\in\WS\}$, and
$\wit\mB=\{(\wth,0)\,|\;(\wth)\in\WP\}$.
\end{teor}
\begin{proof}
It is enough to prove the assertion for $\wit\mB$: see
Theorem \ref{4.teoratract}. Note that the family
\eqref{4.eqlinear} has exponential dichotomy with $F^+=\W\!\times\!\R^2$:
see Remark~\ref{2.notarelac}. This and Remark~\ref{2.notaED}.2 ensure that
$\limsup_{t\to-\infty}r_l(t,\wth,r_0)=\infty$
for all $(\wth)\in\WP$ and $r_0>0$. Hence,
Proposition \ref{4.propacotadas} proves the assertion.
\end{proof}
\vspace{.2cm}\par\noindent
{\bf \hypertarget{B}{B}. The case
$\boldsymbol{\Sigma_A\subset(0,\boldsymbol{\infty})}$.}
Now we will show that, in the three cases of
$\Sigma_A\subset(0,\infty)$ (see Theorem~\ref{2.teorOsel}), the map
$\wit\beta$ is continuous and strictly positive on $\WP$, which means that
$\mA$ can be identified with a \lq\lq solid cylinder" which has
$\W$ as axis.
\begin{teor}\label{4.teorcopiaWS}
Suppose that $\Sigma_A\subset(0,\infty)$. Then,
\begin{itemize}
\item[\rm(i)]
the map $\wit\beta\colon\WP\to\R^+\!$ of Theorem \ref{4.teortapa}
is continuous and strictly positive.
\item[\rm(ii)] The compact set $\{(\wth,\wit\beta(\wth))\,|\;
(\wth)\in\WP\}$ is the global attractor for the flow $\wit\tau$ restricted
to $\WP\times(0,\infty)$. More precisely,
\[
\hspace{1.2cm}\lim_{t\to\infty}(\wit\beta(\wit\sigma(t,\wth))-r(t,\wth,r_0))=0
\quad\text{for all $(\wth)\in\WP$ and $r_0> 0$}\,.
\]
\item[\rm(iii)] The compact set $\Big\{\Big(\w,\lsm\wit\beta(\wth)\sin\theta\\
\wit\beta(\wth)\cos\theta\rsm\Big)\,|\;
(\wth)\in\WP\Big\}$ is the global attractor for the flow $\tau_\R$ restricted
to $\W\times(\R^2-\{\bcero\})$.
\end{itemize}
\end{teor}
\begin{proof}
(i) The property $\Sigma_A\subset(0,\infty)$
ensures that the family of linear systems
\eqref{4.eqlinear} has exponential dichotomy, with $F^-\!=\W\!\times\!\R^2$:
see Remark~\ref{2.notarelac}.
More precisely, there exist $c\ge 1$ and $\gamma>0$ such that
\[
 r_l(t,\wth,r_0)\le c\,e^{\gamma\,t}r_0
 \quad\text{ for all $\,(\wth)\in\WP$, $\,r_0\ge 0$, \,and $\,t\le 0$}\,.
\]
We take $\delta_\rho=\rho/c>0$. Then
$r_l(t,\wth,r_0)\le\rho$ for $t\le 0$ if $r_0\le\delta_\rho$,
so that
\[
 r(t,\wth,r_0)=r_l(t,\wth,r_0)\le c\,e^{\gamma\,t}r_0
 \quad\text{ for all $(\wth)\in\WP\,$ and $\,t\le 0$}\,.
\]
We will deduce from this fact that there exists
$t_\rho>0$ such that
\begin{equation}\label{5.subequi}
 r(t,\wth,\delta_\rho)\ge\delta_\rho \quad\text{ for all
 $\,(\wth)\in\WP\,$ and $\,t\ge t_\rho$}\,.
\end{equation}
Assume for contradiction the existence of sequences $(t_n)\uparrow\infty$ in $\R^+\!$
and $((\w_n,\theta_n))$ in $\WP$ such that
$r(t_n,\w_n,\theta_n,\delta_\rho)=:r_n<\delta_\rho$. Then, on the one hand,
\[
 r(-t_n,\wit\sigma(t_n,\w_n,\theta_n)),\delta_\rho)>
 r(-t_n,\wit\sigma(t_n,\w_n,\theta_n)),r_n)=\delta_\rho
\]
for all $n\in\N$;
and, on the other hand, $0<r(-t_n,\wit\sigma(t_n,\w_n,\theta_n)),\delta_\rho)\le
c\,e^{-\gamma\,t_n}\delta_\rho$, which tends to $0$ as $n$ increases.
This contradiction proves \eqref{5.subequi}. And the same argument proves that,
given any $r_0>0$, there exists $t_{r_0,\rho}$ such that
\begin{equation}\label{5.subequi1}
 r(t,\wth,r_0)\ge\delta_\rho \quad\text{ for all
 $\,(\wth)\in\WP\,$ and $\,t\ge t_{r_0,\rho}$}\,.
\end{equation}
\par
A fundamental consequence derives from \eqref{5.subequi}. Let us define
the sequence of continuous functions $(\alpha_n)$
by
\[
 \alpha_n\colon\WP\to\R^+\!\,,\quad
 (\wth)\mapsto r(nt_\rho,\wit\sigma(-nt_\rho,\wth),\delta_\rho)\,.
\]
Note that all these functions are bounded from below by $\delta_\rho$ and from above
by $\rho$.
We can reason as in the proof of Theorem 3.6 of \cite{nono2} in order to prove: that
\[
 r(t_\rho,\wth,\alpha_n(\wth))\ge\alpha_n(\wit\sigma(t_\rho,\wth))
 \quad\text{ for all $\,(\wth)\in\WP\,$ and $\,n\ge 1$}\,;
\]
that the sequence $(\alpha_n(\wth))_n$ increases with $n$ for all $(\wth)
\in\WP$; that the limit
\[
 \alpha(\wth):=\lim_{n\to\infty}\alpha_n(\wth)
\]
satisfies
\begin{equation}\label{5.alfa}
 \alpha(\wit\sigma(mt_\rho,\wth))=r_l(mt_\rho,\wth,\alpha(\wth))
 \quad\text{ for all $\,m\in\N\,$ and $\,(\wth)\in\WP$}\,;\,
\end{equation}
and that $\alpha$ is a lower semicontinuous map: if $(\wth)=
\lim_{n\to\infty}(\w_n,\theta_n)$, then
$\alpha(\wth)\le\liminf_{n\to\infty}\alpha(\w_n,\theta_n)$.
Note also that $\delta_\rho\le\alpha\le r_\rho$.
\par
On the other hand,
$\wit\beta(\wth)=
\lim_{n\to\infty}r(nt_\rho,\wit\sigma(-nt_\rho,\wth),r_\rho)$
(see the proof of Theorem~\ref{4.teortapa}),
so that $\wit\beta\ge\alpha\ge\delta_\rho>0$. We will check that, as
a matter of fact, they agree, which together with the semicontinuity
properties of $\wit\beta$ and $\alpha$ shows the continuity of
$\wit\beta$ and completes the proof.
\par
We will make this in two steps. In the first one we
assume that $(\WP,\wit\sigma,\R)$ is minimal. Let
$(\w_0,\theta_0)$ be a continuity point of $\wit\beta$.
Theorem 3.6 of \cite{nono2} proves that
\[
\begin{split}
 \mK:&=\cls_{\WP\!\times\!\R^+}\{\wit\tau(t,\w_0,\theta_0,\wit\beta(\w_0,\theta_0))
 \,|\;t\ge 0\}\\
 &=\cls_{\WP\!\times\!\R^+}\{(\wit\sigma(t,\w_0,\theta_0),
 \wit\beta(\wit\sigma(t,\w_0,\theta_0)))\,|\;t\ge 0\}
\end{split}
\]
is a $\wit\tau$-minimal set. The upper semicontinuity of $\wit\beta$
ensures that $r_0\ge r\ge\delta_\rho$ for
all $(\wth,r)\in\mK$. In the words of \cite{nuos4}, the minimal
$\mK$ is strongly above the equilibrium 0. In addition,
\eqref{5.subequi1} also ensures that any
$\wit\tau$-minimal different from that given by 0 is contained in
$\{(\wth,r)\,|\;r\ge \delta_\rho\}\subset\WP\!\times\!\R^+\!$.
Recall that the flow $\wit\tau$ is monotone and
concave: see Remark~\ref{4.notaconcavo}.
In these conditions, Theorem 3.8 of \cite{nuos4} ensures
that $\mK$ is a uniformly exponentially stable copy
of the base which attracts all the forward semiorbits starting above zero.
That is, we can write $\mK=\{(\wth,c(\wth))\,|\;(\wth)\in\WP\}$ for a continuous map
$c\colon\WP\to\R^+$; and given $\ep>0$ there exists $\eta(\ep)>0$
such that if $r_0>0$ and $|r_0-c(\wth)|\le\eta(\ep)$, then
$|r(t,\wth,r_0)-c(\wit\sigma(t,\wth))|\le\ep$ for all $t\ge 0$.
It is very easy to deduce from the
definition of $\mK$ that $c$ and $\wit\beta$ agree
at the continuity points of $\wit\beta$. Now we will check that
$\wit\beta$ and $c$ agree everywhere: we fix $(\wth)\in\WP$ and any $\ep>0$;
choose a sequence $(s_n)\downarrow-\infty$ such that
$(\w_0,\theta_0)=\lim_{n\to\infty}\wit\sigma(s_n,\w,\theta)$;
deduce from $\wit\beta(\w_0,\theta_0)=c(\w_0,\theta_0)$ that
there exists $s_n$ with $|\wit\beta(\wit\sigma(s_n,\w,\theta))-
c(\wit\sigma(s_n,\w,\theta))|\le\eta(\ep)$; and conclude that
\[
 |\wit\beta(\wth)-c(\wth)|=
 |r(-s_n,\ws_n,\wit\beta(\wit\sigma(s_n,\w,\theta)))-
 c(\wit\sigma(-s_n,\wit\sigma(s_n,\w,\theta)))|\le\ep\,,
\]
which means that $\wit\beta(\wth)=c(\wth)$.
Note that this proves that $\wit\beta$ is continuous in the
minimal case; but our goal is to prove that it agrees with $\alpha$,
which will be required in the general case.
\par
The continuity and positiveness of $\wit\beta$ allows us to
find $\eta\in(0,1)$ such that $0<\eta\wit\beta\le\delta_\rho $, so that
$\eta\wit\beta\le\alpha$. In addition, as said in the proof
of Theorem \ref{4.teortapa},
$\eta\,\wit\beta(\wit\sigma(t,\wth))\le r(t,\w,\eta\,\wit\beta(\wth))$
for all $t\ge 0$ and $(\wth)\in\WP$. In these conditions,
Theorem 3.6 of \cite{nono2} and its proof ensure that
the map $\gamma(\wth):=\lim_{t\to\infty}r(t,\wit\sigma(-t,\wth),
\eta\beta(\wit\sigma(-t,\wth)))$ defines a new equilibrium above $0$.
We can repeat all the procedure before performed with $\wit\beta$
in order to conclude that $\gamma$ is continuous. But
Theorem 3.8 of \cite{nuos4} ensures that $\wit\beta$ is the
only strictly positive continuous equilibrium.
That is, $\gamma=\wit\beta$. Consequently,
\[
\begin{split}
 \beta(\wth)&=\lim_{n\to\infty}r(nt_\rho,
 \wit\sigma(-nt_\rho,\wth),\eta\beta(\wit\sigma(-nt_\rho,\wth)))\\
 &\le\lim_{n\to\infty}r(nt_\rho,
 \wit\sigma(-nt_\rho,\wth),\alpha(\wit\sigma(-nt_\rho,\wth)))=\alpha(\wth)\,,
\end{split}
\]
which completes the proof in this case.
\par
Let us now consider the general case. We already know that $\alpha$ and
$\wit\beta$ agree over each $\wit\sigma$-minimal subset $\mM\subset\WP$.
We fix $(\wth)\in\WP$, take a minimal subset $\mM$ contained in its
alpha limit set for $\wit\sigma$,
choose a point $(\bar\w,\bar\theta)$
in $\mM$, take a sequence $(s_n)\downarrow-\infty$ such that
$(\bar\w,\bar\theta)=\lim_{n\to\infty}\wit\sigma(s_n,\wth)$, and assume without
restriction that $s_n=m_nt_\rho-\mu_n$ with $-m_n\in\N$ and
$\mu_n\in[0,t_\rho)$ and that
there exists $\wit\mu:=\lim_{n\to\infty}\mu_n\in[0,t_\rho]$. The continuity of
$\wit\sigma$ allows us to ensure that $\lim_{n\to\infty}\wit\sigma(m_nt_\rho,\wth)=
\wit\sigma(\wit\mu,\bar\w,\bar\theta)=:(\wit\w,\wit\theta)\in\mM$. In particular,
$\alpha(\wit\w,\wit\theta)=\wit\beta(\wit\w,\wit\theta)$,
Now we fix $\ep>0$; take $\delta(\ep)$ satisfying \eqref{4.r};
assume without restriction that the sequences
$(\alpha(\wit\sigma(m_nt_\rho,\wth)))$ and $(\wit\beta(\wit\sigma(m_nt_\rho,\wth)))$
converge; use the inequalities
\[
 0=\wit\beta(\wit\w,\wit\theta)-\alpha(\wit\w,\wit\theta)
 \ge \lim_{n\to\infty}(\wit\beta(\wit\sigma(m_nt_\rho,\wth))-
 \alpha(\wit\sigma(m_nt_\rho,\wth))\ge 0
\]
in order to find $n_*$ such that
\[
 0\le \wit\beta(\wit\sigma(m_{n_*}t_\rho,\wth))-\alpha(\wit\sigma(m_{n_*}t_\rho,\wth))
 \le\delta(\ep)\,;
\]
and combine Theorem~\ref{4.teortapa}(ii), \eqref{5.alfa}  and \eqref{4.r}
to deduce that
\[
\begin{split}
 0\le \wit\beta(\wth)-\alpha(\wth)
 &=r(-m_{n_*}t_\rho,\wit\sigma(m_{n_*}t_\rho,\wth),
 \wit\beta(\wit\sigma(m_{n_*}t_\rho,\wth)))\\
 &\quad-r(-m_{n_*}t_\rho,\wit\sigma(m_{n_*}t_\rho,\wth),
 \alpha(\wit\sigma(m_{n_*}t_\rho,\wth)))
 \le\ep\,.
\end{split}
\]
This completes the proof of the equality also in the general case.
\vspace{.2cm}\par
(ii) Let us take $(\wth,r_0)\in\WP\times(0,\infty)$. Recall that there
exists $t_0$ such that
$r(t,\wth,r_0)\ge\delta_\rho>0$ for all
$(\wth)\in\WP$ and all $t\ge t_0$: see \eqref{5.subequi1}.
We look for a $\wit\sigma$-minimal
set $\mM\subseteq\WP$ contained in the omega limit set of $(\wth)$,
and a sequence $(t_n)\uparrow\infty$ such that there exists
$\lim_{n\to\infty}(\wit\sigma(t_n,\wth),
r(t_n,\wth,r_0))=:(\wit\w,\wit\theta,\wit r)\in\mM\!\times\R^+\!$. In particular,
$\wit r\ge\delta_\rho>0$. We fix $\ep>0$ and take
$\delta(\ep)$ satisfying \eqref{4.r}.
The convergence to $\wit\beta$ in the minimal case
explained in the previous step allows us to take $\wit s>0$ with
$|r(\wit s,\wit\w,\wit\theta,\wit r)-
\wit\beta(\wit\sigma(\wit s,\wit\w,\wit\theta))|\le\delta(\ep)/3$.
And we also look for $t_m$ such that
$|r(t_m+\wit s, \wth,r_0)-r(\wit s,\wit\w,\wit\theta,
\wit r)|\le\delta(\ep)/3$ and $|\wit\beta(\wit\sigma(\wit s+t_m,\wth))-
\wit\beta(\wit\sigma(\wit s,\wit\w,\wit\theta))|\le\delta(\ep)/3$.
These inequalities and \eqref{4.r} ensure that
$|r(t,\wth,r_0)-\wit\beta(\wit\sigma(t,\wth))|\le\ep$ for all $t\ge\wit s+t_m$,
which proves (ii).
\vspace{.2cm}\par
(iii) This last assertion is an immediate consequence of (ii) and
Theorem~\ref{4.teoratract}.
\end{proof}
{\bf \hypertarget{C}{C}. The case $\boldsymbol{\Sigma_A=\{0\}}$.}
 In this case the family \eqref{4.eqlinear} does not admit
exponential dichotomy, so that at least one of its systems admits a nontrivial
bounded solution.
We begin by considering the simplest situation
(which is also the least interesting for the purposes of this paper):
when all the solutions are bounded, the attractor
$\mA$ is homeomorphic to a solid cylinder with
continuous boundary. This is proved in Theorem \ref{5.teortodasacot}.
Before formulating it, we explain some properties
used in its proof and also in Section~\ref{sec6}.
\begin{defi}\label{5.defprimac}
A continuous  function $e\colon\W\to\R$ {\em admits a continuous primitive}
if there exists a continuous function $h_e\colon\W\to\R$ such that
$h_e(\wt)-h_e(\w)=\int_0^t e(\ws)\,ds$ for all $\w\in\W$ and $t\in\R$.
\end{defi}
\begin{nota}\label{6.notasolouno}
If $e$ admits a continuous primitive, then
$\sup_{(t,\w)\in\R\!\times\!\W}\left|\int_0^t e(\ws)\,ds\right|<\infty$, and
Birkhoff's ergodic theorem ensures that
$\int_{\W}e(\w)\,dm_\W=0$.
It is well-known that if $(\W,\sigma,\R)$ is minimal (as in our case)
then $e$ admits a continuous primitive if and only if there exists $\w_0\in\W$ with
$\sup_{t\ge 0}\left|\int_0^t e(\w_0\pu s)\,ds\right|<\infty$ or with
$\sup_{t\le 0}\left|\int_0^t e(\w_0\pu s)\,ds\right|<\infty$:
see e.g.~Proposition A.1 in \cite{jnot}.
\end{nota}
Now we rewrite the matrix $A=\lsm a&b\\c&d\rsm$ of the families
\eqref{4.eqlinear} and \eqref{4.eqR2} as
\begin{equation}\label{5.defe}
 A(\w)=e(\w)\,I_2+\wit A(\w)\quad \text{for
 $\,e(\w):=(1/2)\tr A(\w)$}\,,
\end{equation}
so that $\wit A(\w)=\lsm \wit a&\;b\\c&-\wit a\rsm$. We consider
the family of linear systems with zero trace
\begin{equation}\label{5.eclineartraza01}
 \by'=\wit A(\wt)\,\by
\end{equation}
for $\w\in\W$.
The notation established in Section \ref{sec3} will be used.
\begin{notas}\label{5.notasame}
Recall that we assume $\Sigma_A=\{0\}$. Let $e$ be given by \eqref{5.defe}.
\par
1.~The flows $(\WS,\wih\sigma,\R)$ and
$(\WP,\wit\sigma,\R)$ induced by \eqref{5.eclineartraza01}
agree with those induced by the initial linear family
$\by'=A(\wt)\,\by$: they are defined by \eqref{2.defsigma2} and
\eqref{2.defsigma3}, since the
function $f$ appearing in the angular equations
\eqref{2.eqtheta} and \eqref{3.eqtheta} is the same
(which is due to $a-d=2\wit a$).
\par
2.~Recall that $r_l(t,\wth,1)$ is the solution
of the linear equation \eqref{2.eqrlinear}$_{(\wth)}$ with $r_l(0,\wth,1)=1$, and
let $\wit r_l(t,\wth,1)$ be the solution with $\wit r_l(0,\wth,1)=1$
of the $r$-equation for \eqref{5.eclineartraza01}, namely
\[
 r'=r\,\big(g(\wt,\wit\theta(t,\wth))-e(\wt)\big)=
 r\left(\!-\frac{1}{2}\:\frac{\partial f}{\partial\theta}
 (\wt,\wit\theta(t,\wth))\right).
\]
It is very easy to check that $r_l(t,\wth,1)=\wit r_l(t,\wth,1)\exp\big(\!\int_0^t
e(\ws)\,ds\big)$ and that
$\wit r_l(t,\wth,1)=\exp\big(\!\int_0^t(-1/2)(\partial f/\partial\theta)
(\wit\sigma(s,\wth))\,ds\big)$.
\par
3.~Since $\Sigma_A=\{0\}$, Theorem~\ref{2.teorOsel}
ensures that $\lim_{t\to\infty}(1/t)\ln\det U(t,\w)=0$ for all $\w\in\W$,
and this, relation \eqref{2.reldetU} and
Birkhoff's ergodic theorem yield $\int_\W e(\w)\,dm_\W=0$.
This property, the previous remark and Birkhoff's ergodic theorem yield
$\lim_{t\to\infty}(1/t)\ln r_l(t,\wth,1)=\lim_{t\to\infty}(1/t)\ln
\wit r_l(t,\wth,1)$; and this
and Theorem~\ref{2.teorOsel} ensure
that the Sacker and Sell spectrum of \eqref{5.eclineartraza01}
is also $\{0\}$.
\par
4.~According to \eqref{2.reldetU},
\begin{equation}\label{5.reldetU}
\det U(t,\w_0)=\exp\left(\int_0^t \tr A(\w_0\pu s)\,ds\right)=
\exp\left(\int_0^t 2e(\w_0\pu s)\,ds\right).
\end{equation}
Assume that all the solutions of \eqref{4.eqlinear} are bounded. Then
\eqref{5.reldetU} combined with Theorem A.2 of \cite{jnot} ensures that
$e$ has a continuous primitive, which together with
Remark~\ref{5.notasame}.2
ensures that all the solutions of \eqref{5.eclineartraza01} are also bounded.
\end{notas}
\begin{teor}\label{5.teortodasacot}
Suppose that $\Sigma_A=\{0\}$ and that all the solutions of all the
linear systems \eqref{4.eqlinear} are bounded. Then
the map $\wit\beta\colon\WP\to\R^+\!$ of Theorem \ref{4.teortapa}
is continuous and strictly positive.
\end{teor}
\begin{proof}
Proposition~\ref{4.propacotadas} ensures that $\wit\beta(\wth)>0$ for all
$(\wth)\in\WP$. Let us assume for the moment being that the flow
$(\WP,\wit\sigma,\R)$ is minimal. Then, since all the solutions
of the family of equations $r'=r\,g(\wit\sigma(t,\wth))$
are bounded, there exists a continuous function $h\colon\WP\to(0,\infty)$
such that $(d/dt)h(\wit\sigma(t,\wth))=h(\wth)\,g(\wit\sigma(t,\wth))$:
see e.g.~Proposition A.1 in \cite{jnot}.
Let us call $h_\eta:=\eta\,h$ for $\eta\ge 0$. We look for $\eta>\eta_0$
such that $h_\eta(\wth)\le\rho$ for all $(\wth)\in\WP$, where $\rho$ is the
constant in \eqref{4.defkrho}. Then
$(d/dt)\,h_\eta(\wit\sigma(t,\wth))=h_\eta(\wth)\,\big(g(\wit\sigma(t,\wth))-
k_\rho(h_\eta(\wit\sigma(t,\wth)))\big)$, so that $h_\eta$ determines a copy
of the base for $(\WP\!\times\!\R^+\!,\wit\tau,\R)$ which is
strongly above 0. In addition, this flow is monotone and concave:
see Remark~\ref{4.notaconcavo}. In these conditions, Theorem 3.8(v)
of \cite{nuos4} ensures that $\wit\beta$ is continuous.
\par
We will now check that, in the minimal case, $\wit\beta=h_{\eta_0}$, where
$\eta_0$ is determined by $\sup_{\wth\in\WP}h_{\eta_0}(\wth)=\rho$.
We already know that $h_{\eta_0}\le\wit\beta$. Let us choose $\eta>0$
and check that $\beta<h_\eta$, from where the assertion follows.
For this $\eta$, $(d/dt)\,h_\eta(\wit\sigma(t,\wth))\ge
h_\eta(\wth)\,\big(g(\wit\sigma(t,\wth))-
k_\rho(h_\eta(\wit\sigma(t,\wth)))\big)$, and the inequality
is strict at least at a point $(\wth)\in\WP$. Now we combine
Propositions 4.4 and 4.3 and Theorem 3.6 of \cite{nono2}
in order to deduce that there exists a $\delta>0$ and an equilibrium
$\gamma\le h_\eta-\delta$ such that any point $(\wth,r)$ with
$\gamma(\wth)<r<h_\eta(\wth)$ does not belong to any
copy of the base. A new application of Theorem 3.8(v) of \cite{nuos4}
ensures that $\beta<h_\eta$, as asserted.
It follows easily from $\wit\beta=h_{\eta_0}$ that
\begin{equation}\label{5.igualrho}
\sup_{t\in\R}\wit\beta(\wit\sigma(t,\wth))=\rho\quad\text{for all
$(\wth)\in\WP$}\quad\text{if $\WP$ is $\wit\sigma$-minimal}\,.
\end{equation}
\par
Let us now suppose that $\WP$ is not $\wit\sigma$-minimal. Then the assumed
boundedness of the solutions of the linear family ensures that
$\WP$ is the union of an uncountable family
of $\wit\sigma$-minimal sets, each one of them being an $m$-cover of $\W$
for a common $m\ge 1$. This is proved in the
proof of Theorem 3.1 of \cite{obpa}, since according to Remark~\ref{5.notasame}.4
all the solutions of \eqref{5.eclineartraza01} are bounded.
Note that $\wit\beta$ is continuous over each minimal set.
We define $\alpha(\wth):=\sup_{t\in\R}r_l(t,\wth,1)$ and
deduce from Theorem~\ref{4.teortapa}(ii), the
decomposition of $\WP$ into minimal sets and \eqref{5.igualrho} that
$\alpha(\wth)=(1/\wit\beta(\wth))\,\sup_{t\in\R}\wit\beta(\wit\sigma(t,\wth))=
\rho/\wit\beta(\wth)$. Consequently, $\alpha$ is also continuous
over each minimal set, and the global continuity of $\wit\beta$
will be guaranteed once we have checked that $\alpha$ is continuous on $\WP$.
\par
Note now that
\[
\begin{split}
 &|\alpha(\wth_1)-\alpha(\wth_2)|
 =\left|\,\sup_{t\in\R}\big|U(t,\w)\lsm\sin\theta_1\\\cos\theta_1\rsm\!\big|
 -\sup_{t\in\R}\big|U(t,\w)\lsm\sin\theta_2\\\cos\theta_2\rsm\!\big|\,\right|\\
 &\quad\;\le\sup_{t\in\R}\big|\,U(t,\w)\big(\lsm\sin\theta_1\\\cos\theta_1\rsm
 -\lsm\sin\theta_2\\\cos\theta_2\rsm\!\big)\big|
 \le c\,\big|\!\lsm\sin\theta_1\\\cos\theta_1\rsm-\lsm\sin\theta_2\\\cos\theta_2\rsm\!\big|
 \le \wit c\;\dist_{\P}(\theta_1,\theta_2)
 \end{split}
\]
for all $\w\in\W$ and $\theta_1,\theta_2\in\P$,
where $c:=\sup_{(t,\w)\in\R\!\times\W}|U(t,\w)|<\infty$. (As usual,
$|U|$ represents the Euclidean matrix operator norm.)
Let us take
a sequence $((\w_n,\theta_n))$ in $\WP$ with limit $(\w_0,\theta_0)$. Let
$\mM^0$ be the minimal set containing $(\w_0,\theta_0)$. Then there exists a
sequence $(\w_n,\varphi_n)$ in $\mM^0$ with limit $(\w_0,\theta_0)$.
This assertion can be proved combining the continuity of the map
$\w\mapsto\mM^0_\w$ in the Hausdorff topology of the set of compact subsets
of $\P$ (see e.g.~Theorem 3.3 of \cite{noos4}) with the fact that
$\mM^0_\w$ always contains $m$ elements.
In particular, $\lim_{n\to\infty}\dist_\P(\theta_n,\varphi_n)=0$.
The previous bound ensures that
$\lim_{n\to\infty}|\alpha(\w_n,\varphi_n)-\alpha(\w_n,\theta_n)|=0$,
and the continuity of $\alpha$ on $\mM^0$ yields
$\lim_{n\to\infty}|\alpha(\w_0,\theta_0)-\alpha(\w_n,\varphi_n)|=0$.
Hence, $\lim_{n\to\infty}|\alpha(\w_0,\theta_0)-\alpha(\w_n,\theta_n)|=0$,
which completes the proof.
\end{proof}
Apart from the situation described in Theorem \ref{5.teortodasacot},
the casuistic for $\Sigma_A=\{0\}$ is large. One of the
less trivial cases involves the occurrence of Li-Yorke chaos, which occurs
under some additional conditions. We will explain this
assertion, main goal of the paper, in Section~\ref{sec6}.
\subsection{A simple quasiperiodic example}
\textcolor{black}{
With the aim of clarifying the ideas, we will apply the previous
results to the nonautonomous system
\begin{equation}\label{5.eqquasi}
 \by'=\left[\!\!\begin{array}{cc}\ep&\cos t+\sin(\sqrt{2}t)\\
 -\cos t-\sin(\sqrt{2}t)&\ep\end{array}\right]\by
 - k_{0.5}(|\by|)\,\by\,,
\end{equation}
showing the analogies with
\begin{equation}\label{5.eqaut}
 \by'=\left[\!\!\begin{array}{rr}\ep&1\\-1&\ep\end{array}\right]\by
 - k_{0.5}(|\by|)\,\by\,.
\end{equation}
This second system provides a classical pattern of (autonomous) Hopf bifurcation.
Some simple figures will contribute to the explanation.}
\par
\textcolor{black}{
Let us begin with \eqref{5.eqaut}. Its solutions define a flow on $\R^2$,
which can be understood as a disjoint union of orbits: the projections of the
graphics of the solutions. For $\ep<0$ all the orbits tend (always as $t$ increases)
to the origin $\bcero\in\R^2$. Therefore, this point constitutes the global attractor.
For $\ep>0$ there appears another \lq\lq special\rq\rq~orbit: the circle
of radius $0.5+\sqrt{\ep}$ centered at the origin. It corresponds to the periodic
solutions of the system, and has the property that it attracts any other
orbit excepting the fixed point provided by the origin. Therefore the global attractor
is the closed disk, and its border is the attractor for the flow restricted to the
invariant set $\R^2-\{\bcero\}$. Figure \ref{5.autfig} shows this behaviour.}
\begin{figure}[ht]
\caption{Orbits on $\R^2$, autonomous case \eqref{5.eqaut}. Left: $\ep=-0.15$.
Right: $\ep=0.5$.}
\includegraphics[height=3.5cm]{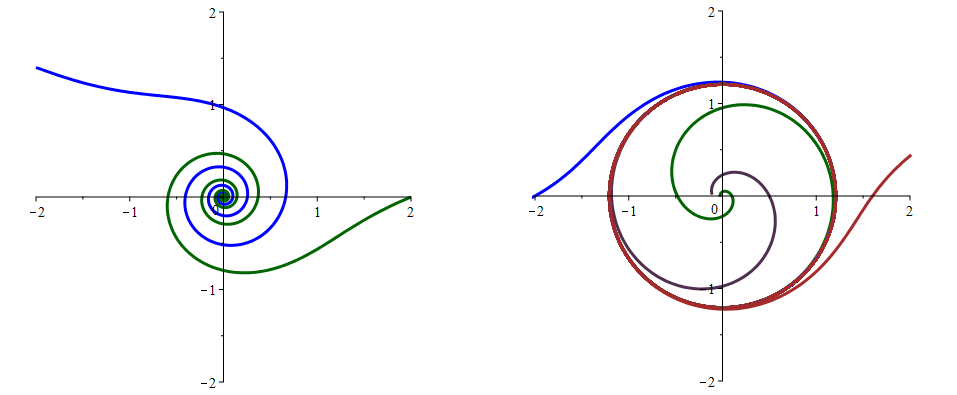}\label{5.autfig}
\end{figure}
\par
\textcolor{black}{
Let us go now with the nonautonomous example, for which the unperturbed linear system
(that corresponding to $\ep=0$) is given by the quasiperiodic matrix-valued function
$A_0(t)=\lsm 0&\cos t+\sin(\sqrt{2}t)\\
-\cos t-\sin(\sqrt{2}t)&0\rsm$. The hull construction described
in the Introduction provides the hull $\T^2$ (the
two-dimensional torus $\T^2$), a Kronecker flow on it,
and the family of systems
\[
 \by'=\left[\!\!\begin{array}{cc}\ep&f(\theta_1+t,\theta_2+\sqrt{2}t)\\
 -f(\theta_1+t,\theta_2+\sqrt{2}t)&\ep\end{array}\right]\by
 - k_{0.5}(|\by|)\,\by
\]
for $(\theta_1,\theta_2)\in\T^2$,
where $f(\theta_1,\theta_2):=\cos(\theta_1)+\sin(\theta_2)$}.
\par
\textcolor{black}{
Since the base flow is minimal, we can determine the Sacker and Sell
spectrum of the linear family by considering just the initial system
(which is given by $(\theta_1,\theta_2)=(0,0)$).
Assume that $\ep=0$. It is easy to check that
all the solutions of the corresponding linear system are bounded,
which ensures that the Sacker and
Sell spectrum is $\{0\}$. And from here it follows easily
that the the Sacker and Sell spectrum of the system corresponding to
$\ep$ is $\{\ep\}$. Therefore, if $\ep<0$, we are in the situation described by
Theorem \ref{4.teorattrzero},
and the global attractor is given by $\T^2\times\{\bcero\}$. And,
if $\ep>0$, the situation fits in that of Theorem \ref{4.teorcopiaWS},
and hence the global attractor
is homeomorphic to a solid cylinder around $\T^2$. Well, in this case it is indeed a
cylinder: $\T^2\times\{\by\in\R^2\,|\;|\by|\le 0.5+\sqrt{\ep}\}$, as one can see
by writting the system in polar coordinates. We have depicted
in Figure \ref{5.quasifig} the projections of the graphics of some solutions for
$\ep<0$ and $\ep>0$. Now they are not orbits of a flow, but the
analogies with the autonomous case are clear. For $\ep<0$ all the
projections approach the origin, which is the section of the
global attractor corresponding to $(\theta_1,\theta_2)=(0,0)$.
And, for $\ep>0$, there appears an \lq\lq invariant\rq\rq~circle, and
the projection of the graphic of any non zero solution approaches
to part of this circle.
}
\begin{figure}[ht]
\caption{Graphic projections, quasiperiodic case \eqref{5.eqquasi}. Left: $\ep=-0.15$.
Right: $\ep=0.5$.}
\includegraphics[height=3.5cm]{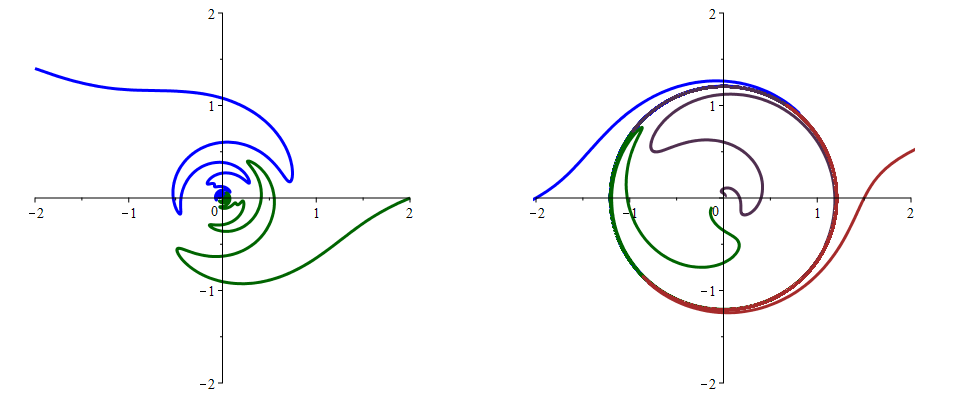}\label{5.quasifig}
\end{figure}
\par
\textcolor{black}{
There is an important difference among the autonomous case and the
quasiperiodic one that we have chosen: the rotation number is 1 for
the first and 0 for the second. But this fact causes no change in
the analysis of the structure of the global attractor:
Figure \ref{5.quasirotfig}, similar
to Figure \ref{5.quasifig}, depicts the projections of the graphics of the
solutions of
\begin{equation}\label{5.eqquasirot}
 \by'=\left[\!\!\begin{array}{cc}\ep&0.5+\cos t+\sin(\sqrt{2}t)\\
 -0.5-\cos t-\sin(\sqrt{2}t)&\ep\end{array}\right]\by
 - k_{0.5}(|\by|)\,\by\,,
\end{equation}
for which the rotation number is 0.5.
}
\begin{figure}[ht]
\caption{Graphic projections, quasiperiodic case with positive rotation
\eqref{5.eqquasirot}. Left: $\ep=-0.15$.
Right: $\ep=0.5$.}
\includegraphics[height=3.5cm]{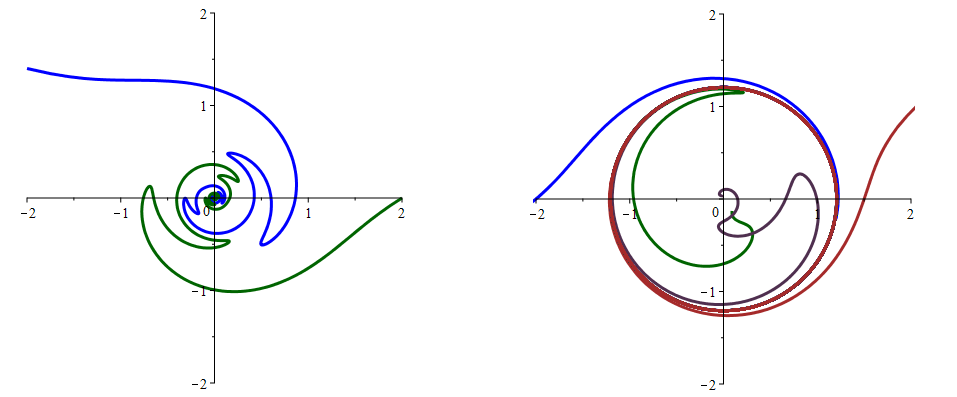}\label{5.quasirotfig}
\end{figure}
\par
\textcolor{black}{
Note finally that, for $\ep=0$, in the autonomous case \eqref{5.eqaut},
the global attractor is the disk of radius $0.5$, so that there is a strong discontinuity with
the situation for $\ep<0$; but also with the situation for $\ep>0$, since
now the border of the disk is no longer the attractor outside the origin:
there is a periodic orbit of radius $r$ for any $r\in(0,0.5)$. Naturally, this is due
to the fact that $k_{0.5}$ vanishes identically on $[0,0.5]$.
And, for the nonautonomous cases \eqref{5.eqquasi} and \eqref{5.eqquasirot},
the situation is similar: according to Theorem~\ref{5.teortodasacot} and
to the expression of the systems in polar coordinates, we know that
in both cases the global attractor is given for $\ep=0$ by the solid cylinder
$\T^2\times\{\by\in\R^2\,|\;|\by|\le 0.5\}$. Therefore, it is again completely
different from that corresponding to $\ep<0$; and also to that appearing
for $\ep>0$ in the sense that now the attractor is composed by infinitely many
compact invariant sets: the cylinders $\T^2\times\{\by\in\R^2\,|\;|\by|\le r\}$
for any $r\in[0,0.5]$. This is the \lq\lq lack of continuity
even in the simplest situation\rq\rq~that we referred to in the Introduction.
Figure \ref{5.cerofig} shows the projections of the graphics of some solutions
of \eqref{5.eqaut}, \eqref{5.eqquasi} and \eqref{5.eqquasirot} for $\ep=0$.
}
\begin{figure}[ht]
\caption{Graphic projections at the bifurcation point $\ep=0$, case
of bounded orbits. Left: autonomous case \eqref{5.eqaut}. Center:
quasiperiodic case with null rotation \eqref{5.eqquasi}
Right: quasiperiodic case with positive rotation \eqref{5.eqquasirot}.}
\includegraphics[height=3.5cm]{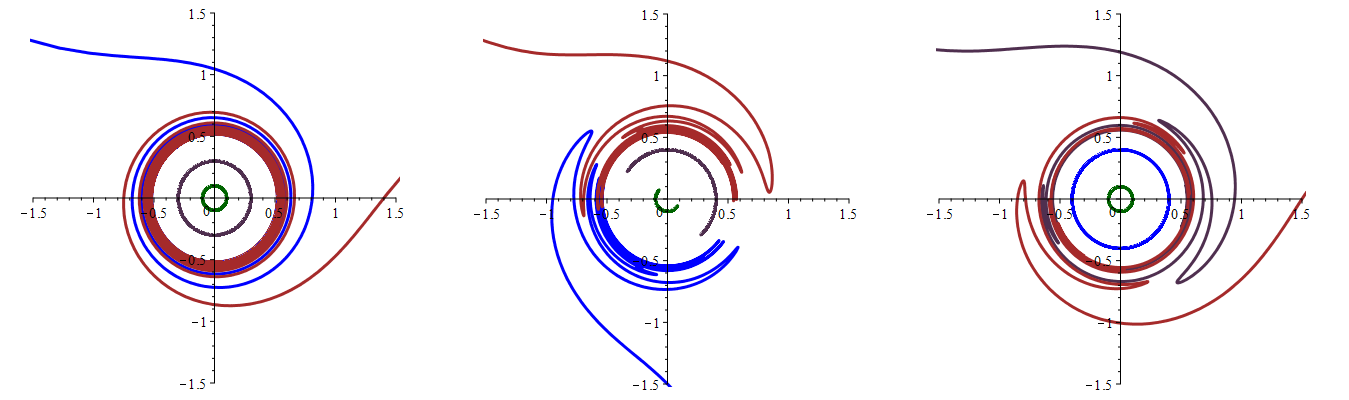}\label{5.cerofig}
\end{figure}
\section{Occurrence of Li-Yorke chaos in the case
of Sacker and Sell spectrum $\{0\}$}\label{sec6}
We continue working with the family \eqref{4.eqR2}
under the conditions
described at the beginning of Section \ref{sec4}, and with the
flows $(\WS,\wih\sigma,\R)$, $(\WP,\wit\sigma,\R)$ and
$(\W\!\times\!\R^2\!,\tau_\R,\R)$ respectively defined by
\eqref{2.defsigma2}, \eqref{2.defsigma3} and \eqref{4.deftaur}.
We will assume in all this section that $\Sigma_A=\{0\}$. In addition,
we will also assume from the beginning that at least one of the
systems of the family \eqref{4.eqlinear} has an unbounded solution:
the alternative has already been analyzed in Theorem~\ref{5.teortodasacot}.
In other words, the set $\W^+\!$ defined by \eqref{4.defOmega+} is not the whole $\W$.
And we will add one more condition:
\begin{hipos}\label{6.H}
The Sacker and Sell spectrum of the linear family \eqref{4.eqlinear}
is $\Sigma_A=\{0\}$, and the set
$\W^+\!$ defined by \eqref{4.defOmega+} satisfies $\W^+\!\ne\W$ and
$m_\W(\W^+\!)=1$.
\end{hipos}
We will make some remarks on these hypotheses at the end of this section.
For now, note that $\W^+\ne\W$ ensures that $\W^+$ is a set of the
first Baire category (see Proposition \ref{4.propBaire}); and
recall that $m_\W(\W^+\!)=1$ is equivalent to say that
the attractor $\mA$ or the flow $(\W\!\times\!\R^2\!,\tau_\R,\R)$
(see Theorem~\ref{4.teoratract})
has positive measure in $\W\!\times\!\R^2$
(see Proposition~\ref{4.propmedidaW+} and Remark~\ref{4.notamedida}).
Our main results, Theorem~\ref{6.teormain} and Corollary \ref{6.coromain},
substitute $\W^+\!\ne\W$ by a stronger condition
in order to guarantee that the flow
$(\W\!\times\!\R^2\!,\tau_\R,\R)$ is Li-Yorke chaotic;
or, more precisely, that its restriction to the
attractor $\mA$ is Li-Yorke chaotic in a quite strong sense.
\par
\par
To prepare the way, we need some new concepts and related
properties. The concept of continuous primitive is explained
in Definition~\ref{5.defprimac}.
\begin{defi}\label{5.defR}
We call $\mR(\W)$ to the set of continuous functions
$e\colon\W\to\R$ with $\int_\W e(\w)\,dm_\W=0$ which do not
admit a continuous primitive and such that
$\sup_{t\le 0}\int_0^t e(\ws)\,ds<\infty$ for $m_\W$-a.e.~$\w\in\W$.
\end{defi}
The next two results are basically proved in
\cite{john9}, \cite{jnot} and \cite{calo}.
We include many details of the proofs for the reader's convenience.
\begin{prop}\label{6.propprevia}
Let $e\colon\W\to\R$ be a continuous function
with $\int_\W e(\w)\,dm_\W=0$. Then
$\sup_{t\le 0}\int_0^t e(\ws)\,ds<\infty$ for $m_\W$-a.e.~$\w\in\W$ if and
only if $\sup_{t\ge 0}\int_0^t e(\ws)\,ds<\infty$ for $m_\W$-almost
every $\w\in\W$.
\end{prop}
\begin{proof}
Let us assume that $\sup_{t\le 0}\int_0^t e(\ws)\,ds<\infty$ for all
$\w$ in a set $\W_0$ with $m_\W(\W_0)=1$, which is clearly
$\sigma$-invariant. Property $\int_\W e(\w)\,dm_\W=0$
combined with the classical recurrence result of
\cite{shne} and with Fubini's theorem
provides a $\sigma$-invariant set $\W_r$ with $m_\W(\W_r)=1$ such that
for all $\w\in\W_r$ there exists $(l_n)\downarrow-\infty$ with
$\lim_{n\to\infty}\int_{l_n}^0 e(\ws)\,ds=0$. Let us define
$\W^e=\W_0\!\cap\W_r$, take $\w_0\in\W^e$, and check that
$\sup_{t\ge 0}\int_0^t e(\w_0\pu s)\,ds<\infty$.
We reason by contradiction assuming the existence of a sequence
$(t_n)\uparrow\infty$
with $\lim_{n\to\infty}\int_0^{t_n}e(\w_0\pu s)\,ds=\infty$.
Since $\w_0\pu t_n\in\W_r$ for all $n\in\N$, we can find $l_n<0$ such that
$1/n>\big|\int_{l_n-t_n}^0e(\w_0\pu(t_n+s))\,ds\,\big|=
\big|\int_{l_n}^{t_n}e(\w_0\pu s)\,ds\,\big|$. But then
$\lim_{n\to\infty}\int_{l_n}^0 e(\w_0\pu s)\,ds=
-\lim_{n\to\infty}\int_0^{t_n} e(\w_0\pu s)\,ds=-\infty$, which
in turn ensures that
$\sup_{t\le0}\int_0^t e(\w_0\pu s)\,ds=\infty$. This is the sought-for contradiction.
The proof of the converse assertion is analogous.
\end{proof}
\begin{prop}\label{6.propR}
Let $e\colon\W\to\R$ be a continuous function with
$\int_\W e(\w)\,dm_0=0$. The following assertions are equivalent:
\begin{itemize}
\item[\rm(1)] $e\in\mR(\W)$.
\item[\rm(2)] There exists a $\sigma$-invariant set
$\W^e\subseteq\W$ with $m_\W(\W^e)=1$
such that, if $\w\in\W^e$, then
$\sup_{t\in\R}\int_0^t e(\ws)\,ds<\infty$,
$\inf_{t\le 0}\int_0^t e(\ws)\,ds=-\infty$, and
$\inf_{t\ge 0}\int_0^t e(\ws)\,ds=-\infty$.
\item[\rm(3)] There exist a measurable function $h_e\colon\W\to(-\infty,0]$
and a $\sigma$-invariant set $\W^e\subseteq\W$ with $m_\W(\W^e)=1$
such that, if $\w\in\W^e$, then
$h_e(\wt)-h_e(\w)=\int_0^t e(\ws)\,ds$ for all $t\in\R$,
$\sup_{t\in\R}h_e(\wt)<\infty$,
$\inf_{t\le 0}h_e(\wt)=-\infty$, and $\inf_{t\ge 0}h_e(\wt)=-\infty$.
\end{itemize}
In addition, in this case, for every $\w$ in a residual subset of $\W$
there exist four sequences $(t_n^i)$ with $(t_n^i)\uparrow\infty$ for $i=1,3$ and
$(t_n^i)\downarrow-\infty$ for $i=2,4$ such that
\[
 \lim_{n\to\infty}\int_0^{t_n^i} e(\ws)\,ds=-\infty  \;\mbox{for $i=1,2$}
 \quad{\mbox{and}}\quad
 \lim_{n\to\infty}\int_0^{t_n^i} e(\ws)\,ds=\infty  \;\mbox{for $i=3,4$}\,.
\]
\end{prop}
\begin{proof}
We begin by assuming that $e\in\mR(\W)$. Then
Proposition \ref{6.propprevia} ensures that
$\sup_{t\in\R}\int_0^t e(\ws)\,ds<\infty$ for all $\w$ in a set
of full measure $\W^e$, which is clearly $\sigma$-invariant.
Hence $\inf_{t\le 0}\int_0^t e(\ws)\,ds=\inf_{t\ge 0}\int_0^t e(\ws)\,ds
=-\infty$ for all $\w\in\W^e$, since otherwise $e$ would admit
a continuous primitive: see Remark \ref{6.notasolouno}.
This completes the proof of (1)$\Rightarrow$(2). The converse implication is
almost obvious: see the first assertion in Remark~\ref{6.notasolouno}.
\par
Let us now assume that (2) holds,
define $h(\w):=-\sup_{t\in\R}\int_0^t e(\ws)\,ds$ for all
$\w\in\W$, and note that
$-\infty\le h(\w)\le 0$ for all $\w\in\W$ and $-\infty<h(\w)\le 0$
for all $\w\in\W^e$. It is not hard to check that
$h(\wt)-h(\w)=\int_0^t e(\ws)\,ds$. We modify
$h$ is order to take the value $0$ on $\W-\W^e$ and
obtain the function $h_e$ of (3). Conversely, it is obvious that
(3) implies (2).
\par
A detailed proof of the last assertion can be found in
Theorem A.2 of \cite{jnot}.
\end{proof}
The relation between the linear families
\eqref{4.eqlinear} and \eqref{5.eclineartraza01},
determined by \eqref{5.defe} and explained in Remarks~\ref{5.notasame},
will be fundamental in what follows.
Let us show the effect of Hypotheses~\ref{6.H}
on the function $e$ of \eqref{5.defe}.
\begin{prop}\label{6.propomega+}
Suppose that Hypotheses \ref{6.H} hold. Then,
the map $e=(1/2)\tr A$ belongs
to the set $\mR(\W)$ of Definition~\ref{5.defR}.
\end{prop}
\begin{proof}
As explained in Remark \ref{5.notasame}.3, Hypotheses \ref{6.H}
ensure that $\int_\W e(\w)\,dm_\W=0$.
Let us fix $\w_0\in\W^+\!$. Proposition~\ref{4.propacotadas} ensures that
all the solutions of the linear system \eqref{4.eqlinear}$_{\w_0}$
are bounded on $\R^-$. This and \eqref{5.reldetU}
yield $\sup_{t\le 0} \int_0^te(\w_0\pu s)\,ds<\infty$. Consequently,
either $e\colon\W\to\R$ admits a continuous primitive, or it belongs to
$\mR(\W)$.
We will check that $\inf_{t\le 0}\det U(t,\w_0)=0$,
which together with \eqref{5.reldetU}
ensures that $\inf_{t\le 0}\int_0^t e(\w_0\pu s)\,ds=-\infty$.
This precludes the
existence of a continuous primitive (see Remark~\ref{6.notasolouno}),
so that $e\in\mR(\W)$.
\par
We assume that $\inf_{t\le 0}\det U(t,\w_0)>0$.
Then $\sup_{s\le 0}|U(s,\w_0)^{-1}|<\infty$. Hence,
$\sup_{t\le 0,s\le 0}|U(t,\w_0\pu s)|=
\sup_{t\le 0,s\le 0}|U(t+s,\w_0)\,U(s,\w_0)^{-1}|<\infty$.
Therefore, since $\{\w_0\pu s\,|\;s\le 0\}$ is dense in $\W$,
$\sup_{t\le 0,\w\in\W}|U(t,\w)|<\infty$. This means that
all the solutions of all the systems \eqref{4.eqlinear}
are bounded on $\R^-$ and hence that
$\W^+\!=\W$ (see once more Proposition~\ref{4.propacotadas}).
But this contradicts one of the Hypotheses~\ref{6.H}.
\end{proof}
Out next purpose is to show that Hypotheses~\ref{6.H}
impose on the family \eqref{5.eclineartraza01} one of the conditions
of Theorem~\ref{3.teorcaos}, part of
whose proof will be fundamental later.
\begin{teor}\label{6.teormedida}
Suppose that Hypotheses~\ref{6.H} holds. Then
the flow $(\WP,\wit\sigma,\R)$ admits an invariant measure which is equivalent
to $m_\W\!\!\times\!l_\P$.
\end{teor}
\begin{proof}
Let $f$ be defined by \eqref{3.deffz}.
According to Remark~\ref{5.notasame}.1 and
Proposition 2.2 of \cite{obpa},
it is enough to look for a function
$p\colon\WP\to\R$ in $L^1(\WP,m_\W\!\!\times\!l_\P)$
with $p(\wth)>0$ for $(m_\W\!\!\times\!l_\P)$-a.a.~$(\wth)$
and such that, for all $(\wth)\in\WP$ and $l\in\R$,
\begin{equation}\label{6.ecfuncional}
 p(\wit\sigma(l,\wth))=p(\wth)\exp\Big(-\int_0^l \frac{\p f}{\p\theta}\,
 (\wit\sigma(s,\wth))\,ds\Big):
\end{equation}
if so, $p$ is the density function of a $\wit\sigma$-invariant measure.
Let us define
\[
 p_t(\wth):=\frac{1}{t}\int_{-t}^0\Big(\exp\int_0^s \frac{\p f}{\p\theta}\,
 (\wit\sigma(r,\wth))\,dr\Big)\,ds\quad\text{~and~}\quad
 p(\wth):=\liminf_{t\to\infty}p_t(\wth)
\]
for $(\wth)\in\WP$.
Our goal is to proof that $p$ satisfies the three mentioned conditions.
We will do it in three steps.
\vspace{.1cm}\par
{\sc Step 1}. We will check that $p\in L^1(\WP,m_\W\!\!\times\!l_\P)$.
More precisely, that
\begin{equation}\label{6.step11}
\int_{\mC} p_t(\wth)\,d(m_\W\!\!\times\!l_\P)=
\frac{1}{t}\int_{-t}^0 (m_\W\!\!\times\!l_\P)(\wit\sigma_s(\mC))\,ds
\end{equation}
for any measurable set $\mC\subseteq\WP$ and $t\ge 0$. This yields
$\int_{\WP} p_t(\wth)\,d(m_\W\!\!\times\!l_\P)=1$, and then,
as a consequence of Fatou's lemma, we obtain
\[
 0\le \int_{\WP}p(\wth)\,d(m_\W\!\!\times\!l_\P)\le
 \liminf_{t\to\infty}\int_{\WP}p_t(\wth)\,d(m_\W\!\!\times\!l_\P)=1\,,
\]
so that $p\in L^1\!(\WP)$.
Let us take a measurable set $\mC\subseteq\WP$. The definition of $p_t$, Fubini's
theorem, and the equality
$\chi_{_\mC}(\wth)=\chi_{_{\wit\sigma_l(\mC)}}(\wit\sigma(l,\wth))$ for
all $l\in\R$ yield
\begin{equation}\label{6.step12}
\begin{split}
&\int_{\WP} p_t(\wth)\,\chi_{_\mC}(\wth)\,d(m_\W\!\!\times\!l_\P)\\
& \quad\;=\frac{1}{t}\int_{-t}^0 \int_\W\int_{\P}
\Big(\exp \int_0^s \frac{\p f}{\p\theta}\,
 (\wit\sigma(l,\wth))\,dl\Big)\,\chi_{_{\wit\sigma_s(\mC)}}
 (\wit\sigma(s,\wth))\,dl_\P\,dm_\W\,ds\,.
\end{split}
\end{equation}
It is easy to deduce from
$(d/dt)\,\wit\theta(t,\wth)=f(\wit\sigma(t,\wth))
=f((\wt,\wit\theta(t,\wth))$
and
$\wit\theta(0,\wth)=\theta$ that
$(d/d\theta)\,\wit\theta(s,\wth)=
\exp\big(\!\int_0^s (\p f/\p\theta)\,(\wit\sigma(l,\wth))\,dl\big)$. In turn,
this implies that
\[
 \int_{\P}\exp\Big(\int_0^s \frac{\p f}{\p\theta}\,
 (\wit\sigma(l,\wth))\,dl\Big)\,\chi_{_{\wit\sigma_s(\mC)}}(\wit\sigma(s,\wth))\,dl_\P=
 \int_{\P}\chi_{_{\wit\sigma_s(\mC)}}(\ws,\theta)\,dl_\P\,.
\]
The $\sigma$-invariance of $m$ ensures that
\[
 \int_{\W}\int_{\P}\chi_{_{\wit\sigma_s(\mC)}}(\ws,\theta)\,dl_\P\,dm_\W=
 \int_{\W}\int_{\P}\chi_{_{\wit\sigma_s(\mC)}}(\wth)\,dl_\P\,dm_\W=
 (m_\W\!\!\times\!l_\P)(\wit\sigma_s(\mC))\,,
\]
and this equality and \eqref{6.step12} show \eqref{6.step11}.
The first step is complete.
\vspace{.1cm}\par
{\sc Step 2}. We will now check that $p(\wth)>0$ for
$(m_\W\!\!\times\!l_\P)$-almost all
$(\wth)\in\WP$.
We use the notation established in Remark~\ref{5.notasame}.2
and the information that it provides in order to write
\[
 p_t(\wth)=\frac{1}{t}\int_{-t}^0 \frac{1}{\wit r_l^{\,2}(s,\wth,1)}\:ds=
 \frac{1}{t}\int_{-t}^0 \frac{1}{r_l^2(s,\wth,1)}\exp\Big(2\int_0^s
 e(\w\pu r)\,dr\Big)\,ds\,.
\]
We write $\exp\big(2\int_0^s e(\w\pu r)\,dr\big)=
H_e(\ws)/H_e(\w)$, for $H_e(\w)=\exp\,(2 h_e(\w))$,
where $h_e$ is the map provided by Propositions \ref{6.propomega+}
and \ref{6.propR}. Note that $H_e(\w)\in(0,1]$. Birkhoff's ergodic
theorem ensures that
the set of points $\w_0$ of $\W^+$ for which there exists
$\lim_{t\to\infty}(1/t)\int_{-t}^0 H_e(\w_0\pu s)\,ds=
\int_\W H_e(\w)\,dm_\W$ has measure $1$. We fix
$\w_0$ in this set  and $\theta_0\in\P$.
Then there exists $c_{(\w_0,\theta_0)}\ge 1$ such that
$r_l^2(t,\w_0,\theta_0,1)\le c_{(\w_0,\theta_0)}$ for all
$t\le 0$: see Proposition~\ref{4.propacotadas}. Altogether,
we have
\[
 p_t(\w_0,\theta_0)\ge \frac{1}{c_{(\w_0,\theta_0)}\,H_e(\w_0)}\;
 \frac{1}{t}\int_{-t}^0 H_e(\w_0\pu s)\,ds\,,
\]
so that
\[
 p(\w_0,\theta_0)\ge \frac{1}{c_{(\w_0,\theta_0)}\,H_e(\w_0)}\int_\W H_e(\w)\,dm_\W>0\,.
\]
\vspace{0cm}\par
{\sc Step 3}. We will finally check that $p$ satisfies \eqref{6.ecfuncional}.
First, we will check that
\begin{equation}\label{6.chain}
\begin{split}
&p_t(\wit\sigma(l,\wth))\,\pu\,\exp\Big(\int_0^l
\frac{\partial f}{\partial\theta}
(\wit\sigma(r,\wth))\,dr\Big)\\
&\qquad\qquad=\frac{t-l}{t}\,p_{t-l}(\wth)+\frac{1}{t}\int_0^l
\,\exp\Big(\int_0^s\frac{\partial f}{\partial\theta}
(\wit\sigma(r,\wth))\,dr\Big)\,ds
\end{split}
\end{equation}
for $0<l<t$. We call $I$ to the left-hand term in the previous expression. Then,
\[
\begin{split}
 I:&=\frac{1}{t}\int_{-t}^0\exp\Big(\int_0^s\frac{\partial f}{\partial\theta}
(\wit\sigma(l+r,\wth))\,dr\Big)\,ds \,\pu\,
\exp\Big(\int_0^l\frac{\partial f}{\partial\theta}
(\wit\sigma(r,\wth))\,dr\Big)\\
&=
\frac{1}{t}\int_{-t}^0\exp\Big(\int_0^{s+l}\frac{\partial f}{\partial\theta}
(\wit\sigma(r,\wth))\,dr\Big)\,ds\\
&=
\frac{t-l}{t-l}\:\frac{1}{t}\int_{-t+l}^l\exp\Big(\int_0^{s}
\frac{\partial f}{\partial\theta}
(\wit\sigma(r,\wth))\,dr\Big)\,ds\\
&=\frac{t-l}{t}\:p_{t-l}(\wth)+
\frac{1}{t}\int_0^l\exp\Big(\int_0^{s}
\frac{\partial f}{\partial\theta}
(\wit\sigma(r,\wth))\,dr\Big)\,ds\,.
\end{split}
\]
We take $\liminf$ in \eqref{6.chain} in order to get
\eqref{6.ecfuncional} for $l>0$. From here we deduce that
\[
\begin{split}
 p(\wit\sigma(-l,\wth))
 &=p(\wth)\,\exp\Big(\int_0^l \frac{\p f}{\p\theta}\,
 (\wit\sigma(s-l,\wth))\,ds\Big)\\
 &=p(\wth)\,\exp\Big(-\int_0^{-l} \frac{\p f}{\p\theta}\,
 (\wit\sigma(s,\wth))\,ds\Big)
\end{split}
\]
if $l>0$,
which completes the proof of \eqref{6.ecfuncional}, and hence that
of the theorem.
\end{proof}
Some more preliminary work is required before formulating and proving the
main result.
Recall that the attractor $\wit\mB$ for $(\WP\times\R,\wit\tau,\R)$ can be written
in terms of the map $\wit\beta$ as
$\wit\mB=\bigcup_{(\wth)\in\WP}\big(\{(\wth)\}\times
[0,\wit\beta(\wth)]\big)$: see Theorem~\ref{4.teortapa}(i).
And recall that $\rho\in(0,1]$ is the constant appearing in \eqref{4.defkrho}.
We define
\begin{align}
 (\WP)_l^+&:=\{(\wth)\in\WP\,|\;0<\sup_{t\in\R}
 \wit\beta(\wit\sigma(t,\wth))\le\rho\}\subseteq(\WP)^+\!\,,\nonumber\\
 \W^*&:=\{\w\in\W\,|\;l_\P(((\WP)_l^+)_\w)=1\}\,,\label{6.defOmega*}\\
 \wit\mB_l&:=\text{closure}_{\WP\times[0,\rho]}
 \{(\wth,r)\in\wit\mB\,|\;(\wth)\in(\WP)_l^+\}\,,\nonumber\\
 \mA_l&:=
 \big\{\big(\w,\lsm r\sin\theta\\r\cos\theta\rsm\big)\,|\;
 (\w,\mkp(\theta),r)\in\wit\mB_l\big\}\,.\nonumber
\end{align}
Note that $\wit\mB_l$ and
$\mA_l$ are compact
subsets of the global attractors $\wit\mB$ (for the flow
$(\WP\!\times\!\R^+,\wit\tau,\R)$)
and $\mA$ (for $(\W\!\times\R^2\!,\tau_\R,\R)$) described in Theorems
\ref{4.teoratract} and \ref{4.teortapa}.
\begin{prop}\label{6.prop413}
Suppose that Hypotheses~\ref{6.H} hold. Then,
\begin{itemize}
\item[\rm(i)] the set $(\WP)_l^+$ is $\wit\sigma$-invariant and
$(m_\W\!\!\times\!l_\P)((\WP)_l^+)=1$.
\item[\rm (ii)] Let $\W^*$ and $\W^+\!$ be defined by
\eqref{6.defOmega*} and \eqref{4.defOmega+}. Then
$\W^*\!$ is $\sigma$-invariant, $m_\W(\W^*)=1$, and
$\W^*\subseteq\W^+\!$.
\item[\rm(iii)]
For every $\w\in\W^*$ there exists $r_\w>0$ such that
$(\wth,r)\in\wit\mB_l$ for all $\w\in\W^*$, $\theta\in\P$ and $r\in[0,r_\w]$.
In particular, the section $(\mA_l)_\w\subset\R^2$
contains the closed disk centered at the origin and with radius
$r_\w>0$ for all $\w\in\W^*$.
\item[\rm(iv)]
$\wit\mB_l$ is $\wit\tau$-invariant,
$\mA_l$ is $\tau_\R$-invariant, and the restrictions of the flows
to these sets are given by the family of linear systems \eqref{4.eqlinear}.
\end{itemize}
\end{prop}
\begin{proof}
The set $(\WP)_l^+$ is composed by $\wit\sigma$-orbits, so that
it is $\wit\sigma$-invariant.
Let $\wit\mu$ be the $\wit\sigma$-invariant measure on $\WP$
provided by Theorem~\ref{6.teormedida}, which we assume to be normalized.
In the case that $\wit\mu$ is ergodic, we can repeat the
arguments of Theorem 35 of \cite{calo} in order to check that
$\wit\mu((\WP)_l^+)=1$ (see also Theorem 5.8 of \cite{clos}).
The theorem of decomposition into ergodic components (see Theorem 6.4
or Corollary 6.5 of \cite{mane}) ensures that this property also
holds in the general case.
The last assertion (i) is an immediate consequence of the first one and
the equivalence of the measures $\wit\mu$ and $m_\W\!\!\times\!l_\P$.
\par
The $\sigma$-invariance of $\W^*$ follows from the fact that $\wit\sigma_(t,\w)$
gives a $C^1$ transformation of $\P$ for each fixed $(t,\w)$, and
the property $m_\W(\W^*)=1$ follows from (i) and Fubini's theorem.
Let us define $\mC:=\{(\wth,r)\in\wit\mB\,|\;(\wth)\in(\WP)_l^+\}$.
We fix $\w\in\W^*$ and choose $\theta_1$ such that $(\wth_1)$ and
$(\wth_2)$ are in $(\WP)_l^+$ for $\theta_2=\theta_1+\pi/2$.
It follows from the definition of $(\WP)_l^+$ and Theorem~\ref{4.teortapa}(ii)
that $\sup_{t\in\R} r(t,\wth_i,\wit\beta(\wth_i))\le \rho$ for
$i=1,2$, so that $r(t,\wth_i,\wit\beta(\wth_i))=r_l(t,\wth_i,\wit\beta(\wth_i))$.
Let us set $r_\w:=(1/\sqrt{2})\inf(\wit\beta(\wth_1),\wit\beta(\wth_2))>0$ and
take any $\theta\in\S$. Let us write
$\lsm\sin\theta\\\cos\theta\rsm=c_1\lsm\sin\theta_1\\\cos\theta_1\rsm
+ c_2\lsm\sin\theta_2\\\cos\theta_2\rsm$, so that $c_1^2+c_2^2=1$.
Since $r_\w\le\rho/\sqrt{2}<\rho$, for $t$ close to $0$,
\[
\begin{split}
 &r(t,\w,\theta,r_\w)=r_l(t,\w,\theta,r_\w)
 =r_\w\big|\,\by_l\big(t,\w,\lsm\sin\theta\\\cos\theta\rsm\!\big)\big|\\
 &\qquad=r_\w\,\Big|\,\by_l\Big(t,\w,c_1\lsm\sin\theta_1\\\cos\theta_1\rsm
 + c_2\lsm\sin\theta_2\\\cos\theta_2\rsm\!\Big)\Big|\\
 &\qquad\le r_\w\,\Bigg(\frac{|c_1|}{\wit\beta(\wth_1)}\,
 r(t,\wth_1,\wit\beta(\wth_1))+
 \frac{|c_2|}{\wit\beta(\wth_2)}\,
 r(t,\wth_2,\wit\beta(\wth_2))\Bigg)\\
 &\qquad\le \frac{|c_1|+|c_2|}{\sqrt{2}}\,\rho\le\rho\,.
\end{split}
\]
An easy contradiction argument shows that $\sup_{t\in\R}
r(t,\w,\theta,r_\w)\le\rho$. This ensures that $\w\in\W^+\!$, so that
(ii) holds. It also ensures that $(\w,\theta,r_\w)\in\wit\mB$ for all
$\theta\in\P$. In particular, $(\w,\theta,r)\in\mC$ for all
$\theta\in\mP$ such that $(\w,\theta)\in(\WP)_l^+$ and all $r\in[0,r_\w]$.
Since $\w\in\W^*$, $l_\P(((\WP)_l^+)_\w)=1$, so
that $((\WP)_l^+)_\w$ is dense in $\mP$. Therefore
$(\w,\theta,r)\in\wit\mB_l$ for all $\theta\in\mP$ and all
$r\in[0,r_\w]$, which is the first assertion in (iii). The second one
follows immediately from the first one and Theorem~\ref{4.teortapa}.
\par
The $\wit\sigma$-invariance of $(\WP)_l^+$ and the $\wit\tau$-invariance of
$\wit\mB$ ensure that the set $\mC$ is $\wit\tau$-invariant,
which in turn guarantees the
invariance of the sets $\wit\mB_l$ and $\mA_l$. And the last
assertion in (iv) follows from
$r(t,\wth,t)\le \rho$ for all $(\wth,r)\in\wit\mB_l$.
\end{proof}
In order to formulate our main result, we define
\begin{equation}\label{6.defOmega0}
 \W^0:=\{\w\in\W\,|\;\wit\beta(\wth)=0\text{ for all $\theta\in\P$}\}
 \subseteq\W-\W^+\!\,.
\end{equation}
It is obvious that if $\W^0$ is nonempty then
the condition $\W^+\!\ne\W$, included in Hypotheses~\ref{6.H}, holds.
But the converse property cannot be guaranteed:
see Remark \ref{4.notaopciones}.1. So, the conditions we have assumed so far
do not guarantee that $\W^0$ is nonempty.
However this stronger requirement will be imposed in
the formulation of our main result. At the end of the section
we will show that these hypotheses are fulfilled in some
situations which are easy to describe.
For now, we recall once again that the condition $(m_\W\!\!\times\!l_{\R_2})(\mA)>0$
is equivalent to $m_\W(\W^+\!)=1$:
see Proposition~\ref{4.propmedidaW+} and Remark~\ref{4.notamedida}.
All this means that Hypotheses~\ref{6.H} are included in those of the next results.
\begin{teor}\label{6.teormain}
Suppose that $\Sigma_A=\{0\}$, that $(m_\W\!\!\times\!l_{\R_2})(\mA)>0$,
and that the set $\W^0$ given by \eqref{6.defOmega0} is nonempty.
Then there exists a $\sigma$-invariant set $\W_0\subseteq\W$ with
$m_\W(\W_0)=1$ such that $l_{\R^2\!}((\mA_l)_\w)>0$ for
all $\w\in\W_0$, and such that any two different
points $\big(\w,\lsm r_1\sin\theta_1\\r_1\cos\theta_1\rsm\!\big),
\big(\w,\lsm r_2\sin\theta_2\\r_2\cos\theta_2\rsm\!\big)$ of $\mA_l$
form a Li-Yorke pair for the flow $\tau_\R$ whenever $\w\in\W_0$.
\end{teor}
\begin{proof}
Let $\lsm r_i\sin\theta_i\\r_i\cos\theta_i\rsm$ for $i=1,2$
be two different points of $(\mA_l)_\w$, and let
us call $\bar\theta_i=\mkp(\theta_i)\in\P$. Note that
$r_i\le\wit\beta(\w,\bar\theta_i)$.
We take $\wit\w_0\in\W^0$ and look for $(t_n)\uparrow\infty$ with
$\wit\w_0=\lim_{n\to\infty}\wt_n$ and such that
there exist $\bar\theta_i^0:=
\lim_{n\to\infty}\wit\theta(t_n,\w,\bar\theta_i)$ for $i=1,2$.
Then
\[
\begin{split}
 &0\le\limsup_{n\to\infty}r_l(t_n,\wth_i,r_i)=
 \limsup_{n\to\infty}r(t_n,\wth_i,r_i)=
 \limsup_{n\to\infty}r(t_n,\w,\bar\theta_i,r_i)\\
 &\qquad\qquad
 \le\limsup_{n\to\infty}r(t_n,\w,\bar\theta_i,\wit\beta(\w,\bar\theta_i))
 =\limsup_{n\to\infty}\wit\beta(\wit\sigma(t_n,\w,\bar\theta_i))\le
 \wit\beta(\wit\w_0,\bar\theta^0_i)=0
\end{split}
\]
for $i=1,2$. Here we have used Proposition \ref{6.prop413}(iv),
\eqref{4.SP}, the monotonicity of the flow $\wit\tau$,
Theorem~\ref{4.teortapa}(ii), the semicontinuity of $\wit\beta$
also ensured by Theorem~\ref{4.teortapa}, and the definition of
$\W^0$. Consequently,
\[
\liminf_{t\to\infty}
\dist_{\R^2\!}\big(\by\big(t,\w,\lsm r_1\sin\theta_1\\r_1\cos\theta_1\rsm\!\big),
\by\big(t,\w,\lsm r_2\sin\theta_2\\r_2\cos\theta_2\rsm\!\big)\big)=0\,.
\]
The goal is hence to find
a $\sigma$-invariant set $\W_0$ with $m_\W(\W_0)=1$ such that
two conditions hold for all $\w\in\W_0$:
$l_{\R^2\!}((\mA_l)_\w)>0$, and
\[
 \limsup_{t\to\infty}
 \dist_{\R^2\!}\big(\by\big(t,\w,\lsm r_1\sin\theta_1\\r_1\cos\theta_1\rsm\!\big),
 \by\big(t,\w,\lsm r_2\sin\theta_2\\r_2\cos\theta_2\rsm\!\big)\big)>0
\]
whenever $\lsm r_1\sin\theta_1\\r_1\cos\theta_1\rsm$ and
$\lsm r_2\sin\theta_2\\r_2\cos\theta_2\rsm$ belong to $(\mA_l)_\w$.
Note that the first condition
will be guaranteed by Proposition \ref{6.prop413}(iii) if we
take (as we will do) $\W_0$ contained in the set $\W^*$
given by \eqref{6.defOmega*}, which satisfies $\W^*\subseteq\W^+\!$  and
$m_\W(\W^*)=1$; and that the second condition is equivalent
to
\begin{equation}\label{6.>0}
\limsup_{t\to\infty}
\dist_{\R^2\!}\big(\by_l\big(t,\w,\lsm r_1\sin\theta_1\\r_1\cos\theta_1\rsm\!\big),
\by_l\big(t,\w,\lsm r_2\sin\theta_2\\r_2\cos\theta_2\rsm\!\big)\big)>0
\end{equation}
whenever $\lsm r_1\sin\theta_1\\r_1\cos\theta_1\rsm$ and
$\lsm r_2\sin\theta_2\\r_2\cos\theta_2\rsm$ belong to $(\mA_l)_\w$,
as Proposition \ref{6.prop413}(iv) ensures.
\par
Recall that our hypotheses guarantee that $m_\W(\W^+\!)=1$ and $\W^+\!\ne\W$.
Hence the map $e=(1/2)\tr A$ satisfies the assertions of
Propositions \ref{6.propomega+} and \ref{6.propR}.
In addition, the Saker and Sell spectrum  of the family of systems
\eqref{5.eclineartraza01} given by $\wit A=A-e\,I_2$ is $\{0\}$
(see Remark \ref{5.notasame}.3). Hence we can distinguish two cases:
either all the solutions of all the systems \eqref{5.eclineartraza01}
are bounded, or the family is in the weakly elliptic case.
\vspace{.1cm}\par
{\sc Case 1}. We begin by assuming that all the solutions of
all the systems \eqref{5.eclineartraza01} are bounded, hypothesis
which we will use twice later.
Remark~\ref{5.notasame}.2 ensures that
\[
 r_l(t,\wth_i,r_i)=r_i\,\wit r_l(t,\wth_i,1)\exp\left(\int_0^t
 e(\ws)\,ds\right)=r_i\,\wit r_l(t,\wth_i,1)\:\frac{\exp(h_e(\wt))}{\exp(h_e(\w))}
\]
for $i=1,2$ if $\w\in\W^*\subseteq\W^+\!$,
where $h_e$ is given by Propositions \ref{6.propomega+}
and \ref{6.propR}. Lusin's theorem
provides a compact set $\mK\subset\W$ with $m_\W(\mK)>0$ such that $h_e$
is continuous on $\mK$, and Birkhoff's ergodic theorem ensures that the
$\sigma$-invariant set $\W_0\subseteq\W^*$ of points
$\w$ for which there exists
$(s_n)\uparrow\infty$ with $\ws_n\in\mK$ for all $n\ge 0$ satisfies
$m_\W(\W_0)=1$. We take $\w\in\W_0$ and
two different points $\lsm r_i\sin\theta_i\\r_i\cos\theta_i\rsm$
of $(\mA_l)_\w$. We also take
$(s_n)\uparrow\infty$ with $\ws_n\in\mK$ such that there
exist $(\w_0,\theta_i^0):=\lim_{n\to\infty}(\ws_n,\wih\theta(s_n,\w,\theta_i))$.
All the solutions of all the systems \eqref{5.eclineartraza01}
are also bounded away from zero:
the determinant of the fundamental matrix solution
of \eqref{5.eclineartraza01} is 1 (see \eqref{2.reldetU}),
so that its inverse is also globally bounded.
Hence the two sequences  $(\wit r_l(s_n,\wth_i,1))$ are
contained in an interval $[\kappa_1,\kappa_2]\subset(0,\infty)$.
So, there is no restriction in assuming that there exist
$\bar r_i:=\lim_{n\to\infty}\wit r_l(s_n,\wth_i,1)>0$, and
there also exist $r_i^*:=\lim_{n\to\infty} r_l(s_n,\wth_i,r_i)=r_i\,\bar r_i\,
\exp(h_e(\w_0))/\exp(h_e(\w))$ with $r_i^*\ne 0$ if $r_i\ne 0$.
\par
The assumed boundedness also ensures that the flow
$(\WS,\wih\sigma,\R)$ (see Remark~\ref{5.notasame}.1)
is distal (see e.g.~the proof of Theorem 3.1 of \cite{obpa}). Consequently,
$\lim_{n\to\infty}\dist_\S\big(\wih\theta(s_n,\wth_1),\wih\theta(s_2,\wth_2)\big)>0$
if $\theta_1\ne\theta_2$, so that \eqref{6.>0} holds in this case.
If, on the contrary, $\theta_1=\theta_2$, we have $r_1\ne r_2$ and
$\bar r_1=\bar r_2$, which
ensures that $r_1^*\ne r_2^*$ and hence that \eqref{6.>0} also holds.
\vspace{.1cm}\par
{\sc Case 2}.
Now we assume that the family \eqref{5.eclineartraza01}
is weakly elliptic. This fact and Theorem~\ref{6.teormedida} ensure
that all the hypotheses of Theorem~\ref{3.teorcaos} hold.
As at the beginning of {\sc step 2} in its proof, we choose a
$\wih\sigma$-minimal set $\mM\subseteq\W\times\S$ and consider the
flow $(\mM\!\times\S,\wih\vartheta_{\mM},\R)$.
We also fix a $\wih\sigma$-ergodic measure $m_\mM$ concentrated on $\mM$.
Let $\mK\subseteq\mM$ be the compact set with $m_\mM(\mK)>1/2$
described in Remark~\ref{3.notapaluego}.
It follows from Proposition~\ref{6.prop413}(i) that
$\mC^0:=(\mK\!\times\S^1)\cap\{(\wth^1\!,\theta)\in\mM\!\times\S\,|\;
(\w,\mkp(\theta))\in(\WP)_l^+\}$
satisfies $(m_\mM\!\!\times\!l_\S)(\mC^0)=m_\mM(\mK)>1/2$.
Let us define
$\wih\beta_*\colon\mM\times\S\to[0,\rho],\;(\wth^1\!,\theta)\mapsto
\wit\beta(\w,\mkp(\theta))$.
Lusin's theorem provides a compact
set $\mC^1\subseteq\mC^0$ with $(m_\mM\!\!\times\!l_\S)(\mC^1)>1/2$
such that the restriction of
$\wih\beta_*$ to $\mC^1$ is continuous.
The definitions of $(\WP)_l^+$
and $\wit\beta$ ensure that $\wih\beta_*(\wth^1\!,\theta)=
\wit\beta(\w,\mkp(\theta))\in(0,\rho]$
for all $(\wth^1\!,\theta)\in\mC^1$, and hence that
$\beta_*:=\inf_{(\wth^1\!,\theta)\in\mC^1}\wih\beta_*(\wth^1\!,\theta)\in(0,\rho]$.
In addition, the set
$\mM_1:=\{(\wth^1)\in\mM\,|\;l_\S(\mC^1_{(\wth^1)})>1/2\}$
satisfies $m_\mM(\mM_1)>0$, since $(m_\mM\!\!\times\!l_\S)(\mC^1)>1/2$.
It is clear that $\mM_1\subseteq\mK$.
It is also clear that, if $(\wth^1)\in\mM_1$,
we can find $\theta_1\in\S$ such that the points $(\wth^1\!,\theta_1)$ and
$(\wth^1\!,\theta_2)$ belong to $\mC^1$ for $\theta_2=\theta_1+\pi/2$.
As in the proof of
Proposition~\ref{6.prop413}(iii), we check that
$\wih\beta_*(\wth^1\!,\theta)\ge\beta^*/\sqrt{2}$ for all $\theta\in\S^1$.
That is,
\begin{equation}\label{6.acabando}
 \wit\beta(\wth)
 \ge\beta^*/\sqrt{2}
 \quad\text{for all $\theta\in\P$
 if there exists $\theta^1\!\in\S$ with $(\wth^1)\in\mM_1$}\,.
\end{equation}
It follows form Birkhoff's ergodic theorem that the $\wih\sigma$-invariant
set $\mM_2$ of points $(\wth^1)\in\mM$ for which
there exists a sequence $(t_n)\uparrow\infty$ with
$\wih\sigma(t_n,\wth^1)\in\mM_1$ for all $n\ge 0$
satisfies $m_\mM(\mM_2)>0$. Since
$m_\mM$ is $\wih\sigma$-ergodic, $m_\mM(\mM_1)=1$.
\par
Now we define $\W_0$ as the projection of $\mM_1$ on $\W$
intersected with the set $\W^*$ defined by \eqref{6.defOmega*},
so that $\W_0$ is $\sigma$-invariant, $m_\W(\W_0)=1$, and
$l_{\R^2\!}((\mA_l)_\w)>0$ for all $\w\in\W_0$.
We fix $\w\in\W_0$,
look for $\theta^1\!\in\S$ with $(\wth^1)\in\mM_2$, and
take $\theta_1,\theta_2\in\S$ and $r_1,r_2\ge 0$ such that
$\big(\w,\lsm r_1\sin\theta_1\\r_1\cos\theta_1\rsm\!\big)$
and $\big(\w,\lsm r_2\sin\theta_2\\r_2\cos\theta_2\rsm\!\big)$
are two different points of $\mA_l$.
Our goal is to prove \eqref{6.>0}.
\par
We fix a sequence of positive numbers $(t_n)\uparrow\infty$
with $\wih\sigma(t_n,\wth^1)\in\mM_1$ for all $n\ge 0$.
Since $\mM_1\subseteq\mK$,
using the notation established in Remark~\ref{3.notapaluego}
we have
\[
 \dist_\S(\wih\theta_\mM(t_n,\wth^1\!,\theta_1),
 \wih\theta_\mM(t_n,\wth^1\!,\theta_2))\ge {\rm d}(\eta_1,\eta_2) \,,
\]
where $\eta_1$ and $\eta_2$ are determined by $(\wth^1\!,\theta_i)\in\mS_{\eta_i}$
for $i=1,2$, and ${\rm d}(\eta_1,\eta_2):=\inf_{(\wth^1)\in\mK}
\dist_\P(\varphi_{\eta_1}(\wth^1),\varphi_{\eta_2}(\wth^1))$.
Here $\wih\theta_\mM$ is the fiber component of the skew-product flow
$(\mM\!\times\S,\wih\vartheta_\mM,\R)$, which agrees with that of
$(\WS,\wih\sigma,\R)$. Hence,
\begin{equation}\label{6.disteta}
 \dist_\S(\wih\theta(t_n,\wth_1),\wih\theta(t_n,\wth_2))\ge {\rm d}(\eta_1,\eta_2)\,.
\end{equation}
On the other hand,
$r_l(t_n,\w,\mkp(\theta_i),r_i)=(r_i/\wit\beta(\w,\mkp(\theta_i))\,
r_l(t_n,\w,\mkp(\theta_i),\wit\beta(\w,\mkp(\theta_i))$
and $r_l(t_n,\w,\mkp(\theta_i),\wit\beta(\w,\mkp(\theta_i))\ge
r(t_n,\w,\mkp(\theta_i),\wit\beta(\w,\mkp(\theta_i))$ (since $g\ge g-k_\rho$
and $t_n>0$). Therefore, Theorem~\ref{4.teortapa}(ii) and \eqref{6.acabando} yield
\[
 \limsup_{n\to\infty}r_l(t_n,\wth_i,r_i)\ge\frac{r_i}{\wit\beta(\w,\mkp(\theta_i))}\,
 \limsup_{n\to\infty}\wit\beta(\wit\sigma(t_n,\w,\mkp(\theta_i)))\ge
 \frac{r_i\,\beta^*}
 {\sqrt{2}\,\wit\beta(\w,\mkp(\theta_i))}>0
\]
if $r_i\ne 0$. Since $r_1$ and $r_2$ cannot be simultaneously 0, there is no restriction
in assuming that
\[
 r_1^\infty:=\limsup_{n\to\infty}r_l(t_n,\wth_i,r_1)>0\,.
\]
Suppose now that $\theta_1\ne\theta_2$. Then $\eta_1\ne\eta_2$, so that
$\rm{d}(\eta_1,\eta_2)>0$ and \eqref{6.disteta} yields
$\dist_\S(\wih\theta(t_n,\wth_1),\wih\theta(t_n,\wth_2))>0$. Hence
\eqref{6.>0} holds unless
$\lim_{n\to\infty}r(t_n,\wth_i,r_i)=0$ for $i=1,2$, which is not the case.
Hence the condition \eqref{6.>0} holds in this case.
If, on the contrary, $\theta_1=\theta_2$, then $r_1\ne r_2$ and
\[
\begin{split}
 &\limsup_{n\to\infty}
 \dist_{\R^2\!}\big(\by_l\big(t_n,\w,\lsm r_1\sin\theta_1\\r_1\cos\theta_1\rsm\!\big),
 \by_l\big(t_n,\w,\lsm r_2\sin\theta_1\\r_2\cos\theta_1\rsm\!\big)\big)\\
 &\qquad=
 \limsup_{n\to\infty}|r_l(t_n,\w,\theta_1,r_1)-r_l(t_n,\w,\theta_1,r_2)|
 =|r_1^\infty-(r_2/r_1)r_1^\infty|>0\,,
\end{split}
\]
so that \eqref{6.>0} also holds in this case. The proof is complete.

\end{proof}
The information provided by Proposition~\ref{6.prop413},
Theorem \ref{6.teormain} and
Proposition~\ref{4.propBaire} leads us to:
\begin{coro}\label{6.coromain}
Suppose that $\Sigma_A=\{0\}$, that $(m_\W\!\!\times\!l_{\R_2})(\mA)>0$,
and that the set $\W^0$ given by \eqref{6.defOmega0} is nonempty.
Then the restricted flow $(\mA,\tau_\R,\R)$ is Li-Yorke fiber-chaotic in measure.
And, in addition, there exists a $\sigma$-invariant residual set of points of $\W$
for which the section $\mA_\w$ is given by $\{\bcero\}$.
\end{coro}
Let us now analyse the hypotheses we have worked with in this
section. All its results require Hypotheses \ref{6.H}.
In particular, according to Proposition~\ref{6.propomega+}, the map
$e=(1/2)\tr A$ must belong to the set $\mR(\W)$ of Definition~\ref{5.defR}.
How restrictive is this condition? The set $\mR(\W)$ is a subset of the set
of continuous functions with mean value zero which admit a measurable
primitive, which is known to be (at least in the case of minimal
almost periodic base flow) of the first Baire category:
this is proved by Johnson in
\cite{john2}, and it is expected to be true in more general settings.
Therefore $\mR(\W)$ is not a \lq\lq large" set, at least from a topological point
of view. As a matter of fact, $\mR(\W)$ is empty in the
autonomous and periodic cases. However it is nonempty for more
complex nonautonomous cases: there are well known examples of
functions $e\in\mR(\W)$. For instance, Johnson shows in
\cite{john9} that this is the case for an example based in
a previous one constructed by Anosov in \cite{anos1}.
The main ideas of the description of this example in \cite{john9}
are used by Ortega and Tarallo in \cite{orta} in order
to optimize the construction of functions with these characteristics.
In both cases, $(\W,\sigma,\R)$ is a minimal quasi periodic
but not periodic flow. And recently Campos {\em et al.} \cite{caot2}
have shown the existence of elements of $\mR(\W)$
whenever the base flow is nonperiodic, uniquely ergodic and minimal.
\par
The condition $e\in\mR(\W)$ does not suffice to guarantee Hypothesis \ref{6.H},
or the more restrictive conditions of Theorem~\ref{6.teormain}
and Corollary \ref{6.coromain}. However,
\begin{prop}\label{6.prophipo}
Suppose that $e\in\mR(\W)$, that $\wit A\colon\W\to\M_{2\times 2}(\R)$
is continuous and with $\tr\wit A=0$, and that all the solutions of all the systems
$\by'=\wit A(\wt)\,\by$ are bounded. Let $k_\rho\colon\R^+\to\R^+$ be given
by \eqref{4.defkrho}. Then the family of systems
\begin{equation}\label{6.ecejemplo}
 \by'=\big(\wit A(\wt)+e(\wt)\,I_2\big)\,\by-k_\rho(|\by|)\,\by
\end{equation}
satisfies the hypotheses of Theorem~\ref{6.teormain} and Corollary \ref{6.coromain}.
\end{prop}
\begin{proof}
Let us call $A:=\wit A+e\,I_2$.
With the notation established in Remark \ref{5.notasame}.2,
$r_l(t,\wth,1)=\wit r_l(t,\wth,1)\exp\big(\!\int_0^t
e(\ws)\,ds\big)$. This relation and
Proposition \ref{6.propR} have two consequences. The first
one is that the solutions of the systems $\by'=A(\wt)\,\by$ are bounded for
all $\w$ in a $\sigma$-invariant set of full measure. In particular,
$\Sigma_A=\{0\}$ (see Theorem~\ref{2.teorOsel}) and
$m_\W(\W^+)=1$ (see Proposition~\ref{4.propacotadas}), so that
$(m_\W\!\!\times\!l_{\R_2})(\mA)>0$ (see Proposition~\ref{4.propmedidaW+} and
Remark~\ref{4.notamedida}). The second
consequence is that all the nonzero solutions of the systems $\by'=A(\wt)\,\by$
are unbounded on $(-\infty,0]$ for $\w$ in a residual subset of $\W$.
(Here we are using that the functions $\wit r_l(t,\wth,1)$ are
bounded away from 0: see e.g.~{\sc case 1} in the proof of
Theorem \ref{6.teormain}.)
This property and Proposition~\ref{4.propacotadas} ensure that
the set $\W^0$ given by \eqref{6.defOmega0} contains a residual set.
\end{proof}
\begin{nota}\label{6.notahipo}
The hypotheses of Theorem~\ref{6.teormain} and Corollary \ref{6.coromain}
are also fulfilled by some families of systems of the form \eqref{6.ecejemplo}
with $e\in\mR(\W)$ and
$\by'=\wit A(\wt)\,\by$ in the weakly elliptic case
(see Definition \ref{3.defiweak}). For instance, we
take $e\in\mR(\W)$ and $\wit A=\lsm e/2&0\\0&-e/2\rsm$, call
$A:=e\,I_2+\wit A$, and note that
\[
 U(t,\w)=\left[\begin{array}{cc}\exp(\int_0^t (3e(\ws)/2)\,ds)&0
 \\0&\exp(\int_0^t (e(\ws)/2)\,ds)\end{array}\right]
\]
is the fundamental matrix solution
of $\by'=A(\wt)\,\by$.
It follows from Theorem~\ref{2.teorOsel} that
$\Sigma_A=\{0\}$. In addition, the solutions of the systems
corresponding to the set $\W^e$ given by Proposition~\ref{6.propR}
are bounded, so that $m_\W(\W^+)=1$; and all the nonzero solutions
of the systems corresponding to a residual subset of $\W$,
also given by Proposition~\ref{6.propR}, are unbounded,
so that the set $\W^0$ is nonempty.
\end{nota}
We conclude by recalling that the carried-on analysis gives rise to
examples fitting in what we called
{\em nonautonomous Hopf bifurcation pattern with zero spectrum\/}
in the Introduction. To construct one of these examples it is
enough to take a family of systems of the form
\eqref{6.ecejemplo} satisfying the hypotheses of Theorem \ref{6.teormain}
and Corollary \ref{6.coromain}
(as those provided by Proposition~\ref{6.prophipo} and Remark~\ref{6.notahipo}), and
to consider the one-parametric family of families of ODEs
\[
  \by'=\big(\wit A(\wt)+(e(\wt)+\ep)\,I_2\big)\,\by-k_\rho(|\by|)\,\by
\]
for $\w\in\W$ and $\ep\in\R$.
The Sacker and Sell spectrum of the associated linear family given by $\ep$
is $\{\ep\}$.
Hence, if $\ep<0$, the global attractor $\mA$ is $\W\times\{\bcero\}$;
and if $\ep>0$, then $\mA$ contains a set $\W\times\{\by\in\R^2\,|\;
|\by|\le\beta\}$ for a $\beta>0$ and, as a matter of fact, it is homeomorphic to
this set. These assertions follow from the discussion carried-on in Section~\ref{sec5}.
Finally, the situation at the bifurcation point $\ep=0$ is that
described in Theorem~\ref{6.teormain} and Corollary \ref{6.coromain},
with the occurrence of Li-Yorke chaos.

\end{document}